\documentclass{amsart}
\usepackage{amssymb}
\usepackage[all]{xy}
\usepackage{mathrsfs}

\newcommand{\Z}{\mathbb{Z}}
\newcommand{\Gm}{\mathbb{G}_{\mathrm{m}}}

\newcommand{\fv}{\mathfrak{v}}
\newcommand{\cN}{\mathcal{N}}
\newcommand{\tcN}{\tilde{\mathcal{N}}}
\newcommand{\cB}{\mathcal{B}}
\newcommand{\fu}{\mathfrak{u}}

\newcommand{\cO}{\mathcal{O}}

\newcommand{\cD}{\mathcal{D}}
\newcommand{\Db}{D^{\mathrm{b}}}
\newcommand{\Dp}{D^+}

\newcommand{\fA}{\mathfrak{A}}
\newcommand{\fB}{\mathfrak{B}}
\newcommand{\fC}{\mathfrak{C}}
\newcommand{\fD}{\mathfrak{D}}
\newcommand{\fT}{\mathfrak{T}}

\newcommand{\ax}{{}^{\Xi}\fA}
\newcommand{\axg}[1]{{}^{#1}\fA}
\newcommand{\aq}[1]{{}^{(\le #1)}\fA}
\newcommand{\aqt}[1]{{}^{(< #1)}\fA}
\newcommand{\dx}{{}^{\Xi}\fD}

\newcommand{\dqt}[1]{{}^{(< #1)}\fD}

\newcommand{\dl}[1]{\fD^{\le #1}}
\newcommand{\dg}[1]{\fD^{\ge #1}}

\newcommand{\de}[1]{\Delta_{#1}}
\newcommand{\na}[1]{\nabla^{#1}}

\newcommand{\tde}[2]{\tilde\Delta_{#1}^{(#2)}}
\newcommand{\tna}[2]{\tilde\nabla^{#1}_{(#2)}}
\newcommand{\dex}[1]{\Delta^{\Xi}_{#1}}

\newcommand{\cS}{\mathscr{S}}

\newcommand{\Coh}{\mathsf{Coh}}
\newcommand{\Pcohg}{\mathsf{PCoh}^G}
\newcommand{\Cohg}{\mathsf{Coh}^G}
\newcommand{\Pcohgm}{\mathsf{PCoh}^{G \times \Gm}}
\newcommand{\Cohgm}{\mathsf{Coh}^{G \times \Gm}}
\newcommand{\uVect}{\underline{\mathsf{Vect}}}
\newcommand{\Rep}{\mathsf{Rep}}
\newcommand{\Repf}{\mathsf{Rep}_{\mathrm{f}}}
\newcommand{\uRep}{\underline{\mathsf{Re\hphantom{\hbox{$\mathsf{a}$}}}}\llap{\hbox{$\mathsf{p}$}}}
\newcommand{\uRepf}{\uRep{}_{\mathrm{f}}}

\newcommand{\Forg}{\mathbb{U}}
\newcommand{\cIC}{\mathcal{IC}}

\newcommand{\cF}{\mathcal{F}}
\newcommand{\cG}{\mathcal{G}}
\newcommand{\cK}{\mathcal{K}}
\newcommand{\cV}{\mathcal{V}}

\numberwithin{equation}{section}
\newtheorem*{thm*}{Theorem}
\newtheorem{thm}{Theorem}[section]
\newtheorem{lem}[thm]{Lemma}
\newtheorem{prop}[thm]{Proposition}
\newtheorem{cor}[thm]{Corollary}
\newtheorem{conj}[thm]{Conjecture}
\theoremstyle{definition}
\newtheorem{defn}[thm]{Definition}
\theoremstyle{remark}
\newtheorem{rmk}[thm]{Remark}

\DeclareMathOperator{\cok}{cok}
\DeclareMathOperator{\Irr}{Irr}
\DeclareMathOperator{\Hom}{Hom}
\DeclareMathOperator{\RHom}{\mathit{R}Hom}

\DeclareMathOperator{\Ext}{Ext}
\DeclareMathOperator{\ind}{ind}
\DeclareMathOperator{\Rind}{\mathit{R}ind}
\DeclareMathOperator{\res}{res}
\DeclareMathOperator{\uHom}{\underline{Hom}}
\DeclareMathOperator{\uRHom}{\mathit{R}\underline{Hom}}

\DeclareMathOperator{\uExt}{\underline{Ext}}

\DeclareMathOperator{\Spec}{Spec}

\DeclareMathOperator{\cRHom}{\mathit{R}\mathcal{H}\mathit{om}}
\DeclareMathOperator{\codim}{codim}
\DeclareMathOperator{\real}{\mathsf{real}}

\newcommand{\D}{\mathbb{D}}
\newcommand{\Lotimes}{\mathchoice%
  {\overset{\scriptscriptstyle L}{\otimes}}%
  {\otimes^{\scriptscriptstyle L}}{\otimes^L}{\otimes^L}}

\DeclareMathOperator{\dom}{\mathsf{dom}}
\DeclareMathOperator{\id}{id}

\newcommand{\la}{\langle}
\newcommand{\ra}{\rangle}
\newcommand{\simto}{\overset{\sim}{\to}}
\newcommand{\hto}{\hookrightarrow}
\newcommand{\thr}{\twoheadrightarrow}

\title[Perverse coherent sheaves in good characteristic]{Perverse coherent sheaves on the nilpotent cone\\ in good characteristic}
\author{Pramod N. Achar}
\address{Department of Mathematics\\
  Louisiana State University\\
  Baton Rouge, LA \ 70803\\
  U.S.A.}
\email{pramod@math.lsu.edu}
\thanks{The author received support from NSF grant DMS-1001594.}

\begin{document}

\begin{abstract}
In characteristic zero, Bezrukavnikov has shown that the category of perverse coherent sheaves on the nilpotent cone of a simply connected semisimple algebraic group is quasi-hereditary, and that it is derived-equivalent to the category of (ordinary) coherent sheaves.  We prove that graded versions of these results also hold in good positive characteristic.
\end{abstract}

\maketitle

%%%%%%%%%%%%%%%%%%%%%%%%%%%%%%%%%%%%%%%%%%%%%%%%%%%%%%%%%%%%%%%%%%%%%%%%%%%
\section{Introduction}
\label{sect:intro}
%%%%%%%%%%%%%%%%%%%%%%%%%%%%%%%%%%%%%%%%%%%%%%%%%%%%%%%%%%%%%%%%%%%%%%%%%%%

Let $G$ be a simply connected semisimple algebraic group over an algebraically closed field $\Bbbk$ of good characteristic.  Let $\cN$ denote the nilpotent variety in the Lie algebra of $G$.  There is a ``scaling'' action of $\Gm$ on $\cN$ that commutes with the $G$-action.  Following~\cite{bez:pc}, we may consider the category of $(G \times \Gm)$-equivariant \emph{perverse coherent sheaves} on $\cN$, denoted $\Pcohgm(\cN)$.  This category has some features in common with ordinary perverse sheaves, but it lives inside the derived category of (equivariant) coherent sheaves.  In this note, we prove the following two homological facts about $\Pcohgm(\cN)$.

\begin{thm}\label{thm:qhered}
The category $\Pcohgm(\cN)$ is quasi-hereditary.
\end{thm}

\begin{thm}\label{thm:dereq}
We have $\Db\Pcohgm(\cN) \cong \Db\Cohgm(\cN)$.
\end{thm}

Theorem~\ref{thm:qhered} means that the category contains a class of distinguished objects, called ``standard'' and ``costandard'' objects, that lead to a kind of Kazhdan--Lusztig theory.  This result was proved in characteristic~$0$ in~\cite{bez:qes}.  (See also~\cite{a:ekt}.)  In fact, the proof given there ``almost'' works in positive characteristic as well; it is quite close to the proof given here.  The same arguments also establish the corresponding result for $\Pcohg(\cN)$, where the $\Gm$-action is forgotten.  

On the other hand, our proof of Theorem~\ref{thm:dereq} makes use of the $\Gm$-action in a crucial way (it means that various $\Ext$-groups carry a grading which we exploit), so it cannot easily be forgotten.  The proof is quite elementary: it relies only on general notions from homological algebra, and it is similar in spirit to the methods of~\cite{bgs}.  Unfortunately, for the moment, these methods seem to be inadequate to prove the following natural analogue of Theorem~\ref{thm:dereq}.

\begin{conj}\label{conj:dereq}
We have $\Db\Pcohg(\cN) \cong \Db\Cohg(\cN)$.
\end{conj}

This conjecture is known to hold in characteristic~$0$ by~\cite{bez:psaf}.  The proof given there involves relating $\Cohg(\cN)$ to perverse sheaves on the affine flag variety $\mathfrak{Fl}$ for the Langlands dual group.  It is likely (and perhaps already known to experts) that a similar approach using \emph{mixed} perverse sheaves would allow one to bring in the $\Gm$-action, leading to a characteristic-$0$ proof of Theorem~\ref{thm:dereq} that is quite different from the one given here.

The reason for the restriction to characteristic~$0$ in~\cite{bez:psaf} is that the arguments there require the base field $\Bbbk$ for $G$ to coincide with the field of \emph{coefficients} of sheaves on $\mathfrak{Fl}$.  The sheaves in~\cite{bez:psaf}, like nearly all constructible sheaves used in representation theory in the past thirty-five years, have their coefficients in $\overline{\mathbb{Q}}_\ell$.  But so-called \emph{modular} perverse sheaves---perverse sheaves with coefficients in a field of positive characteristic---have recently begun to appear in a number of important applications~\cite{fie:sasv, jut:mscdm, jmw:psmrt, soe:ricrt}.  It would be very interesting to develop a sheaf-theoretic approach to Theorem~\ref{thm:dereq} or Conjecture~\ref{conj:dereq} in positive characteristic using modular perverse sheaves.

Finally, we note that in characteristic~$0$, it can be deduced from Theorems~\ref{thm:qhered} and~\ref{thm:dereq} that corresponding results hold for arbitrary connected reductive groups.  In positive characteristic, however, isogenous groups need not have isomorphic nilpotent cones (see, e.g.,~\cite[Remark~2.7]{jan:nort}).  The main theorems depend on key geometric facts about nilpotent cones of simply connected groups.  Although they extend to groups with simply connected derived group, they do not extend to arbitrary reductive groups.

The paper is organized as follows.  Sections~\ref{sect:prelim-ab} and~\ref{sect:der-eq} lay the homological-algebra foundations for the main results, starting with notation and definitions.  The key result of that part of the paper is Theorem~\ref{thm:qexc-dereq}, which states that any quasi-exceptional set satisfying certain axioms gives rise to a derived equivalence.  In Section~\ref{sect:not-red}, we return to the setting of algebraic groups.  Section~\ref{sect:aj} contains a number of technical lemmas on the so-called Andersen--Jantzen sheaves.  The main theorems are proved in Section~\ref{sect:proofs}.

%--------------------------------------------------------------------------
\subsection*{Acknowledgments}
%--------------------------------------------------------------------------

While this project was underway, I benefitted from numerous conversations with A.~Henderson, S.~Riche, and D.~Treumann.  I would also like to express my gratitude to the organizers of the Southeastern Lie Theory Workshop series for having given me the opportunity to participate in the May 2010 meeting.

%%%%%%%%%%%%%%%%%%%%%%%%%%%%%%%%%%%%%%%%%%%%%%%%%%%%%%%%%%%%%%%%%%%%%%%%%%%
\section{Preliminaries on abelian and triangulated categories}
\label{sect:prelim-ab}
%%%%%%%%%%%%%%%%%%%%%%%%%%%%%%%%%%%%%%%%%%%%%%%%%%%%%%%%%%%%%%%%%%%%%%%%%%%

%--------------------------------------------------------------------------
\subsection{Generalities}
\label{subsect:gen}
%--------------------------------------------------------------------------

Fix an algebraically closed field $\Bbbk$.  Throughout the paper, all abelian and triangulated categories will be $\Bbbk$-linear and skeletally small (that is, the class of isomorphism classes of objects is assumed to be a set).  Later, all schemes and algebraic groups will be defined over $\Bbbk$ as well.  For an abelian category $\fA$, we write $\Irr(\fA)$ for its set of isomorphism classes of simple objects.  We say that $\fA$ is a \emph{finite-length category} if it is noetherian and artinian.

Now, let $\fT$ be a triangulated category.  For objects $X, Y \in \fT$, we write
\[
\Hom^i(X,Y) = \Hom(X, Y[i]).
\]
A full subcategory $\fA \subset \fT$ is said to be \emph{admissible} if it stable under extensions and direct summands, and if it satisfies the condition of~\cite[\S 1.2.5]{bbd}.  (Thus, our use of the term ``admissible'' is slightly more restrictive than the definition used in~\cite{bbd}.)  If $\fA \subset \fT$ is admissible, then it is automatically an abelian category, and every short exact sequence in $\fA$ gives rise to a distinguished triangle in $\fT$.  The heart of any $t$-structure on $\fT$ is admissible.  For the following fact, see~\cite[Remarque~3.1.17]{bbd} or~\cite[Lemma~3.2.4]{bgs}.

\begin{lem}\label{lem:ext-hom}
Let $\fA$ be an admissible abelian subcategory of a triangulated category $\fT$.  The natural map
\[
\Ext^i_{\fA}(X,Y) \to \Hom^i_{\fT}(X,Y)
\]
is an isomorphism for $i = 0,1$.  If it is an isomorphism for $i = 0,1, \ldots, k$, then it is injective for $i = k+1$. \qed
\end{lem}

Next, we recall the ``$*$'' operation for objects of a triangulated category $\fD$.  If $\mathcal{X}$ and $\mathcal{Y}$ are classes of objects in $\fD$, then we define
\[
\mathcal{X} * \mathcal{Y} = \left\{ A \in \fD \,\bigg|\, 
\begin{array}{c}
\text{there is a distinguished triangle} \\
\text{ $X \to A \to Y \to $ with $X \in \mathcal{X}$, $Y \in \mathcal{Y}$}
\end{array}
\right\}.
\]
By~\cite[Lemme~1.3.10]{bbd}, this operation is associative.  In an abuse of notation, when $\mathcal X$ is a singleton $\{X\}$, we will often write $X * \mathcal{Y}$ rather than $\{X\} * \mathcal{Y}$.  Note that the zero object is a sort of ``unit'' for this operation.  For instance, we have $\mathcal{X} * \mathcal{Y} * 0 = \mathcal{X} * \mathcal{Y}$.  Given a class $\mathcal{X}$, $\mathcal{X} * 0$ is the class of all objects isomorphic to some object of $\mathcal{X}$. 

%--------------------------------------------------------------------------
\subsection{Tate twist}
\label{subsect:tate}
%--------------------------------------------------------------------------

Many of our categories will be equipped with an automorphism known as a \emph{Tate twist}, and denoted $X \mapsto X\la 1\ra$.  We will always assume that Tate twists are ``faithful,'' meaning that for any nonzero object $X$, we have
\[
X \cong X \la n\ra \qquad\text{if and only if} \qquad n = 0.
\]
A key example is the category $\uVect_{\Bbbk}$ of graded $\Bbbk$-vector spaces, where the Tate twist is the ``shift of grading'' functor.  For $X \in \uVect_{\Bbbk}$, let $X_n$ denote its $n$th graded component.  Then $X\la m\ra$ is the graded vector space given by
\[
(X\la m\ra)_n = X_{n-m}.
\]
We regard $\Bbbk$ itself as an object of $\uVect_{\Bbbk}$ by placing it in degree~$0$.

If $X$ and $Y$ are objects of an additive category equipped with a Tate twist, we let $\uHom(X,Y)$ denote the graded vector space defined by
\[
\uHom(X,Y)_n = \Hom(X,Y\la -n\ra).
\]
Notations like $\uHom^i({-},{-})$, $\uExt^i({-},{-})$, and $\uRHom({-},{-})$ are defined similarly.

The following lemma is a graded analogue of~\cite[Lemma~5]{bez:qes}.

\begin{lem}\label{lem:selfstar}
Let $V$ be an object in $\Dp\uVect_\Bbbk$.
\begin{enumerate}
\item If there are integers $n_1, \ldots, n_k$ such that $0 \in V\la n_1\ra * \cdots V\la n_k\ra$, then $V = 0$.\label{it:selfstar-van}
\item If there are integers $n_1, \ldots, n_k> 0$ such that $\Bbbk \in V * V\la n_1 \ra * \cdots * V\la n_k \ra$, then $H^i(V) = 0$ for $i < 0$, and $H^0(V) \cong \Bbbk$.  For $i > 0$, $H^i(V)$ is concentrated in strictly positive degrees.\label{it:selfstar-pos}

\end{enumerate}
\end{lem}
\begin{proof}
\eqref{it:selfstar-van} Suppose $V \ne 0$, and let $m$ be the smallest integer such that $H^m(V) \ne 0$.  Then, for any object $X \in V\la n_1\ra * \cdots * V\la n_k\ra$, it follows that the map $H^m(V\la n_1\ra) \to H^m(X)$ is injective.  But if $X = 0$, this contradicts the assumption that $H^m(V) \ne 0$.

\eqref{it:selfstar-pos}  The argument given for part~\eqref{it:selfstar-van} shows that $H^i(V) = 0$ for $i < 0$, and that the map $H^0(V) \to H^0(\Bbbk) \cong \Bbbk$ is injective.  Let $Y \in V\la n_1 \ra * \cdots * V\la n_k \ra$ be such that there is a distinguished triangle $V \to \Bbbk \to Y \to$.  If $H^0(V) = 0$, it would follow that $H^0(Y) = 0$, leading to a contradiction with the fact that $H^0(\Bbbk) \ne 0$, so it must be that $H^0(V) \cong \Bbbk$.  We then see from that distinguished triangle that
\[
H^i(V) \cong H^{i-1}(Y) \qquad\text{for all $i \ge 1$.}
\]
Because all the $n_i$ are strictly positive, it follows from the fact that $H^0(V) \cong \Bbbk$ that $H^0(Y)$ is concentrated in strictly positive degrees, and hence so is $H^1(V)$.  Thereafter, we proceed by induction on $i$: if $H^i(V)$ is concentrated in strictly positive degrees, so is $H^i(Y)$, and therefore so is $H^{i+1}(V)$.
\end{proof}

%--------------------------------------------------------------------------
\subsection{Quasi-hereditary categories}
\label{subsect:qhered}
%--------------------------------------------------------------------------

Let $S$ be a set equipped with a partial order $\le$.  Assume that every principal lower set is finite, i.e., that
\begin{equation}\label{eqn:lowfin}
\text{For all $s \in S$, the set $\{ t \in S \mid t \le s \}$ is finite.}
\end{equation}
Let $\fA$ be a finite-length abelian category, and assume that one of the following holds:
\begin{itemize}
\item ``Ungraded case'': There is a fixed bijection $\Irr(\fA) \cong S$.
\item ``Graded case'':  $\fA$ is equipped with a Tate twist, and there is a fixed bijection $\Irr(\fA) \cong S \times \Z$ with the property that for a simple object $L \in \fA$, 
\[
\text{$L$ corresponds to $(s,n)$}
\qquad\text{if and only if}\qquad
\text{$L\la 1\ra$ corresponds to $(s,n+1)$}.
\]
\end{itemize}

In the ungraded case, choose a representative simple object $\Sigma_s$ for each $s \in S$, and let $\aq s$ (resp.~$\aqt s$) be the Serre subcategory of $\fA$ generated by all simple objects $\Sigma_t$ with $t \le s$ (resp.~$t < s$).  

In the graded case, let $\Sigma_s$ denote a representative simple object corresponding to $(s,0) \in S \times \Z$. In this case, $\aq s$ (resp.~$\aqt s$) denotes the Serre subcategory of $\fA$ generated by all simple objects $\Sigma_t\la n\ra$ with $t \le s$ (resp.~$t < s$) and $n \in \Z$.  More generally, for any subset $\Xi \subset S \times \Z$, we let $\ax$ denote the Serre subcategory of $\fA$ generated by the $\Sigma_t\la n\ra$ with $(t,n) \in \Xi$.

In the sequel, we will focus mostly on the graded case.  With the above notation in place, the corresponding definitions and statements for the ungraded cases can usually be obtained simply by omitting Tate twists and by changing ``$\uHom$'' and ``$\uExt$'' to ``$\Hom$'' and ``$\Ext$,'' respectively.  For instance, it is left to the reader to formulate the ungraded version of the following definition.

\begin{defn}\label{defn:qhered}
A category $\fA$ as above is said to be \emph{graded quasi-hereditary} if for each $s \in S$, there is:
\begin{enumerate}
\item an object $\de s$ and a surjective map $\phi_s: \de s \thr \Sigma_s$ such that
\[
\ker \phi_s \in \aqt s
\qquad\text{and}\qquad
\uHom(\de s,\Sigma_t) = \uExt^1(\de s,\Sigma_t) = 0\text{ if $t \not> s$.}
\]
\item an object $\na s$ and an injective map $\psi^s: \Sigma_s \hto \na s$ such that
\[
\cok \psi^s \in \aqt s
\qquad\text{and}\qquad
\uHom(\Sigma_t,\na s) = \uExt^1(\Sigma_t,\na s) = 0\text{ if $t \not> s$.}
\]
\end{enumerate}
Any object isomorphic to some $\de s\la n\ra$ is called a \emph{standard object}, and any object isomorphic to some $\na s\la n\ra$ is a \emph{costandard object}.
\end{defn}

%--------------------------------------------------------------------------
\subsection{Quasi-exceptional sets}
\label{subsect:qexc}
%--------------------------------------------------------------------------

We again let $S$ be a set equipped with a partial order $\le$ satisfying~\eqref{eqn:lowfin}.  Let $\fD$ be a triangulated category, either equipped with a Tate twist (the ``graded case'') or not (the ``ungraded case'').  As noted above, the definitions and lemmas below are usually stated only for the graded case.  However, our first definition comes with a caveat; see the remark below.

\begin{defn}\label{defn:qexc}
A \emph{graded quasi-exceptional set} in $\fD$ is a collection of objects $\{ \na s \}_{s \in S}$ such that the following conditions hold:
\begin{enumerate}
\item If $s \not\ge t$, then $\uHom^i(\na s, \na t) = 0$ for all $i \in \Z$.\label{it:defq-ord}
\item If $i < 0$, then $\uHom^i(\na s, \na s) = 0$, and $\uHom(\na s, \na s) \simeq \Bbbk$.\label{it:defq-self}
\item If $i > 0$ and $n \ge 0$, then $\Hom^i(\na s, \na s \la n\ra)) = 0$.\label{it:defq-mixed}
\item The objects $\{ \na s\la n\ra \mid s \in S,\ n \in \Z \}$ generate $\fD$ as a triangulated category.\label{it:defq-gen}
\end{enumerate}
For $s \in S$, we denote by $\dqt s$ the full triangulated subcategory of $\fD$ generated by all $\na t\la n\ra$ with $t < s$.
\end{defn}

\begin{rmk}\label{rmk:ungr-qexc}
An \emph{ungraded quasi-exceptional set} is defined with analogues of conditions~\eqref{it:defq-ord}, \eqref{it:defq-self}, and~\eqref{it:defq-gen} above, but \emph{without} condition~\eqref{it:defq-mixed}.  The omission of condition~\eqref{it:defq-mixed} makes the two cases substantially different.  In particular, the results of Section~\ref{sect:der-eq} apply only to the graded case.
\end{rmk}

\begin{defn}\label{defn:dual-exc}
A graded quasi-exceptional set $\{ \na s \}$ in $\fD$ is said to be \emph{dualizable} if for each $s \in S$, there is an object $\de s$ and a morphism $\iota_s: \de s \to \na s$ such that:
\begin{enumerate}
\item The cone of $\iota_s$ lies in $\dqt s$.\label{it:defdq-cone}
\item If $s > t$, $\uHom^i(\de s, \na t) = 0$ for all $i \in \Z$.\label{it:defdq-van}
\end{enumerate}
The set $\{ \de s \}$ is known as the \emph{dual quasi-exceptional set}.
\end{defn}

It follows from the second condition that $\uHom(\de s, X) = 0$ for all $X \in \dqt s$.  The proofs of the following two lemmas about a dual set are routine; we omit the details.

\begin{lem}
If $\{ \na s\}$ is a dualizable quasi-exceptional set, then the members of the dual set $\{ \de s \}$ are uniquely determined up to isomorphism. \qed
\end{lem}

\begin{lem}\label{lem:dual-exc}
Let $\{ \na s \}$ be a dualizable quasi-exceptional set, and let $\{ \de s \}$ be its dual set.  Then:
\begin{enumerate}
\item If $s \not\le t$, then $\uHom^i(\de s, \de t) = 0$ for all $i \in \Z$.\label{it:lemq-ord}
\item If $i < 0$, then $\uHom^i(\de s, \de s) = 0$, and $\uHom(\de s, \de s) \simeq \Bbbk$.\label{it:lemq-self}
\item If $i > 0$ and $n \ge 0$, then $\Hom^i(\de s, \de s\la n\ra) = 0$.\label{it:lemq-mixed}
\item The objects $\{ \na s\la n\ra \mid s \in S,\ n \in \Z \}$ generate $\fD$ as a triangulated category.\label{it:lemq-gen}
\end{enumerate}
Furthermore, for all $i \in \Z$, there are natural isomorphisms
\[
\uHom^i(\de s, \de s) \overset{\sim}{\to} \uHom^i(\de s, \na s) \overset{\sim}{\rightarrow} \uHom^i(\na s, \na s). \qed
\]
\end{lem}

\begin{defn}
A dualizable quasi-exceptional set $\{\na s\}_{s \in S}$ with dual set $\{\de s\}_{s \in S}$, is said to be \emph{abelianesque} if we have
\[
\uHom^i(\na s,\na t) = \uHom^i(\de s,\de t) = 0
\qquad\text{for all $i< 0$.}
\]
\end{defn}

The main technical result we need about quasi-exceptional sets is the following.  

\begin{thm}\label{thm:qexc-t}
Let $\fD$ be a triangulated category with a Tate twist, and let $\{\na s\}_{s \in S}$ be an abelianesque dualizable quasi-exceptional set with dual set $\{\de s\}_{s \in S}$.  The categories
\begin{align*}
\dl 0 &= \{ X \in \fD \mid \text{$\Hom(X, \na s\la n\ra [d] = 0$ for all $n \in \Z$ and all $d < 0$} \}, \\
\dg 0 &= \{ X \in \fD \mid \text{$\Hom(\de s\la n\ra [d], X) = 0$ for all $n \in \Z$ and all $d > 0$} \}
\end{align*}
constitute a bounded $t$-structure on $\fD$.  In addition, its heart $\fA = \dl 0 \cap \dg 0$ has the following properties:
\begin{enumerate}
\item $\fA$ contains all $\de s\la n\ra$ and $\na s\la n\ra$.
\item There is a natural bijection $\Irr(\fA) \simto S \times \Z$; the simple object $\Sigma_s$ corresponding to $(s,0) \in S \times \Z$ is the image of the map $\iota_s: \de s \to \na s$. 
\item $\fA$ is a finite-length, graded quasi-hereditary category; the $\de s\la n\ra$ are the standard objects, and the $\na s\la n\ra$ are the costandard objects.
\end{enumerate}
\end{thm}

\begin{proof}[Proof sketch]
A similar statement in the ungraded case, with the ``abelian\-esque'' condition omitted, is proved in~\cite[Propositions~1 and~2]{bez:qes}.  In {\it loc.~cit.}, the standard and costandard objects in the heart are ${}^t H^0(\de s)$ and ${}^t H^0(\na s)$, where ${}^t H^0({-})$ denotes cohomology with respect to the $t$-structure in the statement of the theorem.  But the abelianesque condition clearly implies that the $\de s$ and $\na s$ already lie in the heart of the $t$-structure, so the ungraded version of the theorem follows from the aforementioned results.  The same arguments work in the graded case as well; cf.~\cite[Proposition~4]{bez:ctm}.
\end{proof}

%--------------------------------------------------------------------------
\subsection{Projective covers}
\label{subsect:constr-proj}
%--------------------------------------------------------------------------

We end this section with a result that lets us construct projectives in an abelian category starting from projectives in a Serre subcatgory.  Its proof is similar to that of~\cite[Theorem~3.2.1]{bgs}.

\begin{prop}\label{prop:constr-proj}
Let $\fA$ be a finite-length abelian category.  Let $L \in \fA$ be a simple object with a projective cover $M$, and let $R$ denote the kernel of $M \thr L$.  Let $\fB \subset \fA$ be the Serre subcategory of objects that do not have $L$ as a subquotient.  Let $L' \in \fB$ be a simple object.  Assume that the following conditions hold:
\begin{enumerate}
\item We have $\Hom(M,M) \simeq \Bbbk$.
\item Inside $\fB$, $L'$ admits a projective cover $P'$.
\item $\fA$ and $\fB$ are admissible subcategories of a triangulated category $\fT$, and
\begin{equation}\label{eqn:hom2-van}
\Hom^2_{\fT}(M,R) = \Hom^2_{\fT}(P',R) = 0.
\end{equation}
\end{enumerate}
Then $L'$ admits a projective cover $P$ in $\fA$, arising in a short exact sequence
\begin{equation}\label{eqn:constr-proj}
0 \to \Ext^1_{\fA}(P', M)^* \otimes M \to P \to P' \to 0.
\end{equation}
\end{prop}
\begin{proof}
Let $E = \Ext^1(P',M)$, and consider the identity map $\id: E \to E$ as an element of $\Hom(E,E)$.  Following this element through the chain of isomorphisms
\[
\Hom(E,E) \simeq E^* \otimes E \simeq E^* \otimes \Ext^1(P',M) \simeq \Ext^1(P', E^* \otimes M),
\]
we obtain a canonical element $\nu \in \Ext^1(P', E^* \otimes M)$.  Form the short exact sequence corresponding to $\nu$, and define $P$ to be its middle term.  We have thus constructed the sequence~\eqref{eqn:constr-proj}.  We must now show that $P$ is a projective cover of $L'$.

Because $\Hom(M,M) \simeq \Bbbk$, we have natural isomorphisms
\[
\Hom(E^* \otimes M, M) \simeq E \otimes \Hom(M,M) \simeq E.
\]
Consider now the following commutative diagram:
\[
\xymatrix{
\Hom(E^* \otimes M,M) \ar[r]^-{{-}\circ \nu} \ar[d]_{\wr} & \Ext^1(P', M) \ar@{=}[d] \\
E \ar[r]^{\id} & E}
\]
We see that the natural map $\Hom(E^* \otimes M,M)) \to \Ext^1(P', M)$ is an isomorphism, so from the long exact sequence associated to~\eqref{eqn:constr-proj}, we find that the map
\begin{equation}\label{eqn:pl-hom1}
\Hom(P', M) \to \Hom(P, M)
\end{equation}
is an isomorphism as well, and that
\[
\Ext^1(P, M) \to \Ext^1(E^* \otimes M, M)
\]
is injective. But $\Ext^1(M, M) = 0$ because $M$ is projective, so
\begin{equation}\label{eqn:pl-ext1}
\Ext^1(P, M) = 0.
\end{equation}

Before proceeding, we observe that $R \in \fB$.  Otherwise, if $R$ had a subquotient isomorphic to $L$, there would be a nonzero map $M \to R$, since $M$ is the projective cover of $L$.  But the composition $M \to R \hto M$ yields a nonscalar element of $\Hom(M,M)$, a contradiction.

Next, for any object $X \in \fB$, we claim that
\[
\Hom(E^* \otimes M, X) = \Ext^1(E^* \otimes M, X) = 0.
\]
The former holds because $L$, the unique simple quotient of $M$, does not occur in any composition series for $X$, and the latter holds because $M$ is projective.  Then, from the long exact sequence obtained by applying $\Ext^i(\cdot, X)$ to~\eqref{eqn:constr-proj}, we obtain the following isomorphisms for any $X \in \fB$:
\begin{gather}
\Hom(P', X) \overset{\sim}{\to} \Hom(P, X), \label{eqn:pl-hom2} \\
\Ext^1(P', X) \overset{\sim}{\to} \Ext^1(P, X). \notag
\end{gather}
Since $P'$ is a projective object of $\fB$, the latter isomorphism actually implies
\begin{equation}
\Ext^1(P', X) = \Ext^1(P, X) = 0.\label{eqn:pl-ext2}
\end{equation}
Form the following commutative diagram with exact rows:
\[
\xymatrix{
\Hom(P', R) \ar[r] \ar[d] &
\Hom(P', M) \ar[r] \ar[d] &
\Hom(P', L) \ar[r] \ar[d] &
\Ext^1(P', R) \ar[d] \\
\Hom(P, R) \ar[r] &
\Hom(P, M) \ar[r] &
\Hom(P, L) \ar[r] &
\Ext^1(P, R)
}
\]
The first vertical map is an isomorphism by~\eqref{eqn:pl-hom2}, and the second by~\eqref{eqn:pl-hom1}.  Both terms in the fourth column vanish by~\eqref{eqn:pl-ext2}.  Thus, the third vertical map is also an isomorphism.  But $\Hom(P',L) = 0$, so
\[
\Hom(P,L) = 0
\]
as well.  Combining this with~\eqref{eqn:pl-hom2} and the fact that $L'$ is the unique simple quotient of $P'$, we see that it is the unique simple quotient of of $P$ as well.

Moreover, from~\eqref{eqn:pl-ext2}, we see that $\Ext^1(P,X) = 0$ for all $X \in \fB$.  To prove that $P$ is a projective object of $\fA$, it remains only to show that $\Ext^1(P,L) = 0$.  Using Lemma~\ref{lem:ext-hom}, we may form the long exact sequence
\[
\cdots \to \Ext^1(P,M) \to \Ext^1(P,L) \to \Hom^2_{\fT}(P,R) \to \cdots.
\]
We saw in~\eqref{eqn:pl-ext1} that the first term vanishes, and the assumption~\eqref{eqn:hom2-van} implies that the last term does as well.  Thus, $\Ext^1(P,L) = 0$, as desired.
\end{proof}

\section{Derived equivalences from quasi-exceptional sets}
\label{sect:der-eq}
%%%%%%%%%%%%%%%%%%%%%%%%%%%%%%%%%%%%%%%%%%%%%%%%%%%%%%%%%%%%%%%%%%%%%%%%%%%

In this section, $\fD$ will be a triangulated category equipped with a Tate twist and an abelianesque dualizable graded quasi-exceptional set $\{\na s\}_{s \in S}$ with dual set $\{\de s\}$, with $S$ satisfying~\eqref{eqn:lowfin}.  Let $\fA$ denote the heart of $t$-structure on $\fD$ as in Theorem~\ref{thm:qexc-t}.  Under mild assumptions, there is a natural $t$-exact functor of triangulated categories 
\[
\real: \Db\fA \simto \fD,
\]
called a \emph{realization functor}. For a construction of $\real$ in various settings, see~\cite{ar:kdsf, bei:dcps,bbd}.  The goal of this section is to prove that under an additional assumption (the ``effaceability property'' of Section~\ref{subsect:efface}), this is an equivalence of categories.  

Below, Sections~\ref{subsect:qstd}--\ref{subsect:convex} contain a number preparatory results.  The derived equivalence result, Theorem~\ref{thm:qexc-dereq}, is proved in Section~\ref{subsect:main}.

%--------------------------------------------------------------------------
\subsection{Standard filtrations and quasistandard objects}
\label{subsect:qstd}
%--------------------------------------------------------------------------

We begin with a number of technical lemmas on the existence and properties of certain objects which are filtered by standard objects.  Most of the results of this section are trivial in the case where the quasi-exceptional set is actually \emph{exceptional}, meaning that $\uHom^i(\na s, \na s) = 0$ for $i > 0$.

\begin{defn}
Let $X \in \fA$.  A filtration
\[
0 = X_0 \subset X_1 \subset \cdots \subset X_k = X
\]
is called a \emph{standard filtration} if there are elements $s_1, \ldots, s_k \in S$ and integers $n_1, \ldots, n_k \in \Z$ such that $X_i/X_{i-1} \simeq \de {s_i}\la n_i\ra$ for each $i$.  If $X$ has such a filtration with $s_1 = \cdots = s_k = s$, $X$ is said to be \emph{$s$-quasistandard}.  The notions of \emph{costandard filtration} and \emph{$s$-quasicostandard} are defined similarly.
\end{defn}

\begin{defn}
The \emph{standard order} is the partial order $\preceq_\Delta$ on $S \times \Z$ given by
\[
(s,n) \preceq_\Delta (t,m)
\qquad\text{if $s < t$, or else if $s = t$ and $n \ge m$.}
\]
Similarly, the \emph{costandard order} $\preceq_\nabla$ is given by
\[
(s,n) \preceq_\nabla (t,m)
\qquad\text{if $s < t$, or else if $s = t$ and $n \le m$.}
\]
A member of a subset $\Xi \subset S \times \Z$ is said to be \emph{standard-maximal} (resp.~\emph{costandard-maximal}) if it is a maximal element of $\Xi$ with respect to $\preceq_\Delta$ (resp.~$\preceq_\nabla$).
\end{defn}

A number of statements in this section, starting with the following lemma, contain both a ``standard'' part and a ``costandard'' part.  In each instance, we will only prove the part pertaining to standard objects.  It is, of course, a routine matter to adapt these arguments to the costandard case.

For the maps $\phi_s: \de s \to \Sigma_s$ and $\psi^s: \Sigma_s \to \na s$ as in Definition~\ref{defn:qhered}, we introduce the notation
\[
R_s = \ker \phi_s, \qquad Q_s = \cok \psi^s.
\]

\begin{lem}\label{lem:m-proj}
If $(s,n)$ is standard-maximal in $\Xi$, then $\de s\la n\ra$ is a projective cover of $\Sigma_s\la n\ra$ in $\ax$.  If $(s,n)$ is costandard-maximal, then $\na s\la n\ra$ is an injective hull of $\Sigma_s\la n\ra$.
\end{lem}
\begin{proof}
We already know that $\de s\la n\ra$ has $\Sigma_s\la n\ra$ as its unique simple quotient, and that $\Ext^1(\de s\la n\ra, \Sigma_t\la m\ra) = 0$ whenever $t \not\ge s$.  To prove that $\de s\la n\ra$ is projective in $\ax$, it remains to show that
\[
\Ext^1(\de s\la n\ra,\Sigma_s\la m\ra) = 0 \qquad\text{if $m \ge n$.}
\]
Consider the short exact sequence
\[
0 \to R_s\la m\ra \to \de s\la m\ra \to \Sigma_s\la m\ra \to 0.
\]
Since $R_s\la m\ra \in \aqt s$, we have $\Hom^i(\de s\la n\ra, R_s\la m\ra) = 0$ for all $i \ge 0$.  It follows that there is an isomorphism
\[
\Ext^1(\de s\la n\ra, \de s\la m\ra) \overset{\sim}{\to} \Ext^1(\de s\la n\ra, \Sigma_s\la m\ra).
\]
When $m \ge n$, Lemma~\ref{lem:dual-exc}\eqref{it:lemq-mixed} tells us that $\Ext^1(\de s\la n\ra, \de s\la m\ra) = 0$, as desired.
\end{proof}

\begin{prop}\label{prop:qstd}
For each $k \ge 0 $, there is an $s$-quasistandard object $\tde sk$ with the following properties:
\begin{enumerate}
\item $\Hom(\tde sk, \Sigma_s) \simeq \Bbbk$, and $\Hom(\tde sk, \Sigma_t\la m\ra) = 0$ if $(t,m) \ne (s,0)$.\label{it:qstd-quot}
\item $\Ext^1(\tde sk, \Sigma_t\la m\ra) = 0$ if $(t,m) \preceq_\Delta (s,-k)$.\label{it:qstd-proj}
\item $\tde sk$ has a standard filtration whose subquotients are various $\de s\la m\ra$ with $-k \le m \le 0$.\label{it:qstd-nosub}
\end{enumerate}
Similarly, there is an $s$-quasicostandard object $\tna sk$ such that
\begin{enumerate}
\item $\Hom(\Sigma_s, \tna sk) \simeq \Bbbk$, and $\Hom(\Sigma_t\la m\ra, \tna sk) = 0$ if $(t,m) \ne (s,0)$.
\item $\Ext^1(\Sigma_t\la m\ra, \tna sk) = 0$ if $(t,m) \preceq_\nabla (s,k)$.
\item $\tna sk$ has a costandard filtration whose subquotients are various $\na s\la m\ra$ with $0 \le m \le k$.
\end{enumerate}
\end{prop}
\begin{proof}
We proceed by induction on $k$.  When $k = 0$, we set $\tde s0 = \de s$. For this object, parts~\eqref{it:qstd-quot} and~\eqref{it:qstd-proj} are contained in Definition~\ref{defn:qhered}, and part~\eqref{it:qstd-nosub} is trivial.  Suppose now that $k > 0$, and that we have already defined $\tde s{k-1}$ with the desired properties.  Let
\[
\Xi = \{ (t,m) \mid (t,m) \preceq_\Delta (s,-k) \},
\]
and let $\Psi = \Xi \smallsetminus (s,-k)$.

We will define $\tde sk$ by invoking Proposition~\ref{prop:constr-proj}, but we must first check that the hypotheses of that proposition are satisfied.  Parts~\eqref{it:qstd-quot} and~\eqref{it:qstd-proj} say that $\tde s{k-1}$ is a projective cover of $\Sigma_s$ in $\axg {\Psi}$.  By Lemma~\ref{lem:m-proj}, we have that $\de s\la -k\ra$ is a projective cover of $\Sigma_s\la -k\ra$ in $\ax$.  Moreover, because $R_s\la -k\ra \in \dqt s$, we know from Definition~\ref{defn:dual-exc}\eqref{it:defdq-van} that
\[
\Hom^2(\de s\la -k\ra, R_s\la -k\ra) = \Hom^2(\tde s{k-1}, R_s\la -k\ra) = 0.
\]
(The latter vanishing holds because $\tde s{k-1}$ is $s$-quasistandard.)  

Let $\tde sk$ be the projective cover of $\Sigma_s$ in $\ax$ obtained from Proposition~\ref{prop:constr-proj}.  It is clear then that parts~\eqref{it:qstd-quot} and~\eqref{it:qstd-proj} of the proposition hold for $\tde sk$.  Moreover, we have an exact sequence
\[
0 \to \Ext^1(\tde s{k-1}, \de s\la -k\ra)^* \otimes \de s\la -k\ra \to \tde sk \to \tde s{k-1} \to 0
\]
from which we can see that part~\eqref{it:qstd-nosub} holds as well.
\end{proof}

%--------------------------------------------------------------------------
\subsection{Effaceability}
\label{subsect:efface}
%--------------------------------------------------------------------------

For the remainder of Section~\ref{sect:der-eq}, we assume that the quasi-exceptional set $\{\na s\}$ has the following additional property.

\begin{defn}\label{defn:efface}
An abelianesque quasi-exceptional set $\{\na s\}_{s \in S}$ is said to have the \emph{effaceability property} if the following two conditions hold:
\begin{enumerate}
\item For any morphism $f: X[-d] \to \de s$ where $d > 0$ and $X$ is $s$-quasistandard, there is an object $Y \in \aq s$ and an injective map $g: \de s \hto Y$ such that $g \circ f = 0$.\label{it:defeff-std}
\item For any morphism $f: \na s \to X[d]$ where $d > 0$ and $X$ is $s$-quasicostandard, there is an object $Y \in \aq s$ and a surjective map $h: Y \thr \na s$ such that $f \circ h = 0$.\label{it:defeff-costd}
\end{enumerate}
\end{defn}

\begin{lem}\label{lem:efface}
For any morphism $f: X[-d] \to \de s$ where $d > 0$ and $X$ is $s$-quasistandard, there is an $s$-quasistandard object $Y$ and an injective map $g: \de s \hto Y$ such that $g \circ f = 0$.  Moreover, every standard subquotient of $Y/g(\de s)$ is isomorphic to some $\de s\la m\ra$ with $m > 0$.
\end{lem}
\begin{proof}
Let $g: \de s \hto Y$ be an embedding as in Definition~\ref{defn:efface}.  We must show how to replace this $Y$ by a certain kind of $s$-quasistandard object.  For now, we know only that $Y \in \aq s$.  This means that $Y/g(\de s)$ has a filtration with simple subquotients lying in $\aq s$.  We may write:
\begin{equation}\label{eqn:ystar1}
Y \in \de s * \Sigma_{t_1}\la p_1\ra * \Sigma_{t_2}\la p_2\ra * \cdots * \Sigma_{t_k}\la p_k\ra,
\end{equation}
with $t_i \le s$ for all $i$.  From the distinguished triangle $\de s \to \Sigma_s \to R_s[1] \to$, we have
\[
\Sigma_s\la m\ra \in \de s\la m\ra * R_s[1]\la m\ra.
\]
For each factor $\Sigma_{t_i}\la p_i\ra$ in~\eqref{eqn:ystar1} with $t_i = s$, let us replace it by $\de s\la p_i\ra * R_s[1]\la p_i\ra$.  We will then have
\begin{multline}\label{eqn:ystar2}
Y \in \de s * I_1 * \cdots * I_l \\
\text{where each $I_i$ is one of: } 
\begin{cases}
\de s\la m\ra & \text{for some $m \in \Z$,} \\
R_s[1]\la m\ra & \text{for some $m \in \Z$, or} \\
\Sigma_t\la m\ra & \text{for some $t < s$ and some $m \in \Z$.}
\end{cases}
\end{multline}
Note that each factor $I_i$ that is not of the form $\de s\la m\ra$ belongs to $\dqt s$.  Now, for $I \in \dqt s$, $\Hom(\de s\la m\ra,I[1]) = 0$ by Definition~\ref{defn:dual-exc}\eqref{it:defdq-van}, so any distinguished triangle $I \to J \to \de s\la m\ra \to$ splits.  In other words, $I * \de s\la m\ra$ contains only the isomorphism class of the direct sum $I \oplus \de s\la m\ra$, and in particular, we have
\[
I * \de s\la m\ra \subset \de s\la m\ra * I 
\qquad\text{if $I \in \dqt s$.}
\]
Using this fact, we can rearrange the expression~\eqref{eqn:ystar2} so that all factors of the form $\de s\la m\ra$ occur to the left of all factors in $\dqt s$.  In other words, we may assume without loss of generality that~\eqref{eqn:ystar1} reads as follows:
\[
Y \in \de s * \de s\la m_1\ra * \cdots * \de s\la m_k\ra * I_{k+1} * \cdots * I_l
\qquad\text{with $I_i \in \dqt s$ for $i \ge k+1$.}
\]
This means that there is a distinguished triangle
\begin{equation}\label{eqn:ystardt}
Y' \to Y \to I \to
\end{equation}
with
\begin{equation}\label{eqn:ystar3}
Y' \in \de s * \de s\la m_1\ra * \cdots * \de s\la m_k\ra
\qquad\text{and}\qquad
I \in I_{k+1} * \cdots * I_l \subset \dqt s.
\end{equation}

Recall that $\Hom(\de s[k]\la m\ra, I) = \Hom(\de s[k]\la m\ra, I[1]) = 0$ for any $k,m \in \Z$, by Definition~\ref{defn:dual-exc}\eqref{it:defdq-van}.  It follows that $\Hom(X[-d], I) = \Hom(X[-d]), I[1]) = 0$.  From the long exact sequence obtained by applying $\Hom(\de s,{-})$ and $\Hom(X[-d],{-})$ to~\eqref{eqn:ystardt}, we obtain natural isomorphisms
\[
\Hom(\de s, Y') \overset{\sim}{\to} \Hom(\de s, Y)
\qquad\text{and}\qquad
\Hom(X[-d], Y') \overset{\sim}{\to} \Hom(X[-d], Y).
\]
The first of these shows that $g: \de s \hto Y$ factors in a unique way through $Y'$.  Let $g': \de s \to Y'$ be the induced map.  Now, the fact that $g \circ f = 0$ means that $g' \circ f$ is in the kernel of $\Hom(X[-d], Y') \to \Hom(X[-d], Y)$, so the second isomorphism above shows that $g' \circ f = 0$.  However, $g'$ does not necessarily have the other properties claimed in the proposition.

To repair this, we use the ``rearrangement'' method again.  By Lemma~\ref{lem:dual-exc}\eqref{it:lemq-mixed}, we have $\Hom(\de s\la m\ra, \de s[1]\la n\ra) = 0$ if $n > m$, so it follows that
\[
\de s\la n\ra * \de s\la m\ra \subset \de s\la m\ra * \de s\la n\ra \qquad\text{if $n > m$.}
\]
Using this to rearrange terms in~\eqref{eqn:ystar3}, we may first assume without loss of generality that $m_1 \le m_2 \le \cdots \le m_k$, and then we may write
\[
Y' \in \de s\la m_1\ra * \cdots * \de s\la m_j\ra * \de s * \de s\la m_{j+1}\ra * \cdots * \de s\la m_k\ra
\]
where $m_1 \le \cdots \le m_j \le 0 < m_{j+1} \le \cdots \le m_k$.  In other words, we have a distinguished triangle
\[
J \to Y' \to Y'' \to
\]
with $J \in \de s\la m_1\ra * \cdots * \de s\la m_j\ra$ and $Y'' \in \de s * \de s\la m_{j+1}\ra * \cdots * \de s\la m_k\ra$.  Let $g'': \de s \to Y''$ be the natural map.  Clearly, $g'' \circ f = 0$.  It remains only to check that $g''$ is injective.  Recall that a map in $\fA$ is injective if and only if its cone in $\fD$ actually lies in the abelian category $\fA$ (and in that case, the cone is the cokernel).  By construction, the cone of $g''$ lies in $\de s\la m_{j+1}\ra * \cdots * \de s\la m_k\ra \subset \fA$, as desired.
\end{proof}

\begin{prop}\label{prop:tde-vanish}
Let $k \ge 0$ and $d \ge 1$.  If $n < k+d$, then
\[
\Hom^d(\tde sk\la n\ra, \de s) = \Hom^d(\na s, \tna sk\la-n\ra) = 0.
\]
\end{prop}
\begin{proof}
We proceed by induction on $d$.  Suppose first that $d = 1$.  Note that every composition factor of $\de s\la -n\ra$ is some $\Sigma_t\la m\ra$ with $(t,m) \preceq_\Delta (s, -n)$.  The assumption that $n < k+1$ means that $(s,-n) \preceq_\Delta (s,-k)$, so in view of Lemma~\ref{lem:ext-hom}, the result follows from Propostion~\ref{prop:qstd}\eqref{it:qstd-proj}.

Now, suppose $d > 1$.  Given $f \in \Hom^d(\tde sk\la n\ra, \de s)$, choose an embedding $g: \de s \to Y$ as in Lemma~\ref{lem:efface}, and let $Z = Y/g(\de s)$.  Consider the exact sequence
\[
\Hom^{d-1}(\tde sk\la n\ra, Z) \to \Hom^d(\tde sk\la n\ra, \de s) \to \Hom^d(\tde sk\la n\ra, Y).
\]
For any $m > 0$, we have $n-m < k+d-1$, so it follows from the inductive assumption that
\[
\Hom^{d-1}(\tde sk\la n\ra, \de s\la m\ra) \simeq \Hom^{d-1}(\tde sk\la n-m\ra, \de s) = 0 \qquad\text{if $m > 0$.}
\]
Recall from Lemma~\ref{lem:efface} that the standard subquotients of $Z$ are all various $\de s\la m\ra$ with $m > 0$.  It follows that $\Hom^{d-1}(\tde sk\la n\ra, Z) = 0$, so the map
\[
\Hom^d(\tde sk\la n\ra, \de s) \to \Hom^d(\tde sk\la n\ra, Y)
\]
is injective.  The morphism $f$ is in the kernel of this map (because $g \circ f = 0$), so $f = 0$.  Thus, $\Hom^d(\tde sk\la n\ra, \de s) = 0$, as desired.
\end{proof}

\begin{prop}\label{prop:tde-l-vanish}
If $t < s$, or else if $t = s$ and $n-m < k+d$, then
\[
\Hom^d(\tde sk\la n\ra, \Sigma_t\la m\ra) = \Hom^d(\Sigma_t\la n\ra, \tna sk\la m\ra) = 0.
\]
\end{prop}
\begin{proof}
If $t < s$, this follows from Definition~\ref{defn:dual-exc}\eqref{it:defdq-van} and the fact that $\tde sk\la n\ra$ is $s$-quasistandard.  If $t = s$, consider the exact sequence
\[
\Hom^d(\tde sk\la n\ra, \de s\la m\ra) \to \Hom^d(\tde sk\la n\ra, \Sigma_s\la m\ra) \to \Hom^{d+1}(\tde sk\la n\ra, R_s\la m\ra).
\]
The first term vanishes by Proposition~\ref{prop:tde-vanish}, and the last again because $\tde sk\la n\ra$ is $s$-quasistandard and $R_s\la m\ra \in \aqt s$.  Therefore, $\Hom^d(\tde sk\la n\ra, \Sigma_s\la m\ra) = 0$.
\end{proof}

\begin{cor}\label{cor:tde-na-vanish}
If $t \ne s$, or else if $t = s$ and $n-m < k+d$, then we have $\Hom^d(\tde sk\la n\ra, \na t\la m\ra) = \Hom^d(\de s\la n\ra, \tna sk\la m\ra) = 0$. \qed
\end{cor}

%--------------------------------------------------------------------------
\subsection{Subcategories associated to convex sets}
\label{subsect:convex}
%--------------------------------------------------------------------------

The next step towards our theorem is to show that certain Serre subcategories of $\fA$ with finitely many simple objects have enough projectives, and that these projectives have desirable $\Hom$-vanishing properties in $\fD$.

\begin{defn}\label{defn:convex}
A subset $\Xi \subset S \times \Z$ is said to be \emph{convex} if the following two conditions hold:
\begin{enumerate}
\item For each $(s,n) \in \Xi$, we have $\de s\la n\ra \in \ax$ and $\na s\la n\ra \in \ax$.
\item For any $s \in S$, the set of integers $\{ n \in \Z \mid (s,n) \in \Xi\}$ is either empty or an interval $\{a_0, a_0 + 1, \ldots, a_0 + k\}$.
\end{enumerate}
\end{defn} 

\begin{defn}\label{defn:xi-std}
Let $\Xi \subset S \times \Z$ be a finite convex set.  Let $(s,n) \in \Xi$, and let $a_s = \min \{ m \in \Z \mid (s,m) \in \Xi \}$.  Let
\[
\dex s = \tde s{n-a_s}.
\]
The object $\dex s\la n\ra$ is said to be \emph{$\Xi$-standard}, or the \emph{$\Xi$-standard cover} of $\Sigma_s\la n\ra$.

A \emph{$\Xi$-standard filtration} of an object $X \in \ax$ is a filtration each of whose subquotients is a $\Xi$-standard object.
\end{defn}

\begin{lem}\label{lem:exist-convex}
Every finite subset of $S \times \Z$ is contained in a finite convex set.
\end{lem}
\begin{proof}
Given a finite set $\Xi \subset S \times \Z$, let $F_0(\Xi) \subset S$ be the set of $s \in S$ such that one of the following conditions holds:
\begin{itemize}
\item There is an $n \in \Z$ with $(s,n) \in \Xi$ but either $\de s\la n\ra \notin \ax$ or $\na s\la n\ra \notin \ax$.
\item There are integers $a<b<c$ such that $(s,a), (s,c) \in \Xi$ but $(s,b) \notin \Xi$.
\end{itemize}
Let $F(\Xi)$ be the lower closure of $F_0(\Xi)$, i.e.,
\[
F(\Xi) = \{s \in S \mid \text{there is some $s_0 \in F_0(\Xi)$ such that $s \le s_0$} \}.
\]
It is obvious that $F_0(\Xi)$ is finite, and then it follows from~\eqref{eqn:lowfin} that $F(\Xi)$ is finite as well.  Of course, $\Xi$ is convex if and only if $F(\Xi)$ is empty.  

We prove the lemma by induction on the size of $F(\Xi)$.  If $F(\Xi) \ne \varnothing$, let $s \in F(\Xi)$ be a maximal element (with respect to $\le$).  Then $s \in F_0(\Xi)$.  Let us put
\[
a_0 = \min \{n \mid (s,n) \in \Xi\}, \qquad
b_0 = \max \{n \mid (s,n) \in \Xi\}.
\]
Finally, let
\[
\Xi' = \Xi \cup \left\{ (t,m) \,\Bigg|\,
\begin{array}{c}
\text{$\Sigma_t\la m\ra$ occurs as a composition factor in} \\
\text{some $\de s\la r\ra$ or $\na s\la r\ra$ with $a_0 \le r \le b_0$}
\end{array}
\right\}.
\]
It is clear that $s \notin F(\Xi')$.  Note that all the new pairs $(t,m) \in \Xi' \smallsetminus \Xi$ have $t \le s$.  It follows that every element $s' \in F_0(\Xi')$ either belongs to $F_0(\Xi)$ or is${}<s$.  Therefore, $F(\Xi') \subset F(\Xi) \smallsetminus \{s\}$, so by induction, $\Xi'$, and therefore $\Xi$, is contained in a finite convex set.
\end{proof}

\begin{lem}\label{lem:xi-std}
If $\Xi \subset S \times \Z$ is a finite convex set, and $(s,n) \in \Xi$, then $\dex s\la n\ra$ belongs to $\ax$ and has $\Sigma_s\la n\ra$ as its unique simple quotient.  Moreover, for all $d \ge 1$, we have
\begin{equation}\label{eqn:xi-std}
\begin{aligned}
\Hom^d(\dex s\la n\ra, \Sigma_t\la m\ra) &= 0
&\qquad&\text{if $(t,m) \in \Xi$ and $t \le s$,} \\
\Hom^d(\dex s\la n\ra, \na t\la m\ra) &= 0
&\qquad&\text{for all $(t,m) \in \Xi$.}
\end{aligned}
\end{equation}
In particular, if $(s,n)$ is a costandard-maximal element of $\Xi$, then $\dex s\la n\ra$ is a projective cover of $\Sigma_s\la n\ra$.
\end{lem}
\begin{proof}
Let $a_s$ be as in Definition~\ref{defn:xi-std}.  Because $\Xi$ is convex, the standard objects $\de s\la a_s\ra, \de s\la a_s+1\ra, \ldots, \de s\la n\ra$ all belong to $\ax$, so it follows by Proposition~\ref{prop:qstd} that $\dex s\la n\ra = \tde s{n-a_s}\la n\ra$ does as well.  We also already know that $\Sigma_s\la n\ra$ is the unique simple quotient of $\dex s\la n\ra$.  Finally, note that if $(s,m) \in \Xi$, then $a_s \le m$, so $n - m < (n - a_s) +d$.  Therefore,~\eqref{eqn:xi-std} is an immediate consequence of Proposition~\ref{prop:tde-l-vanish} and Corollary~\ref{cor:tde-na-vanish}.
\end{proof}

\begin{prop}\label{prop:fc-proj}
Let $\Xi \subset S \times \Z$ be a finite convex set.  Every simple object $\Sigma_s\la n\ra \in \ax$ admits a projective cover $P$ with a $\Xi$-standard filtration.  Moreover, we have
\begin{equation}\label{eqn:fc-proj}
\Hom^d(P, X) = 0
\qquad\text{for all $d \ge 1$ and all $X \in \ax$.}
\end{equation}
\end{prop}
\begin{proof}
We proceed by induction on the size of $\Xi$.  If $\Xi = \varnothing$, there is nothing to prove.  Otherwise, let $(s,n)$ be a costandard-maximal element of $\Xi$, and let $\Psi = \Xi \smallsetminus \{(s,n)\}$.  Lemma~\ref{lem:xi-std} tells us that $\dex s\la n\ra$ is a projective cover of $\Sigma_s\la n\ra$ satisfying~\eqref{eqn:fc-proj}.

Next, let $R$ denote the kernel of the map $\dex s\la n\ra \thr \Sigma_s\la n\ra$.  $R$ has a composition series consisting of $R_s\la n\ra$ and various $\de s\la m\ra$ with $r \le m < n$.  We see thus that $R$ is contained in $\axg {\Psi}$.

Consider a pair $(t,m) \in \Psi$.  We assume inductively that $\Sigma_t\la m\ra$ admits a projective cover $P'$ in $\axg {\Psi}$ with a $\Psi$-standard filtration and satisfying~\eqref{eqn:fc-proj}.  In particular, we have $\Hom^2(P',R) = 0$.  We have already seen above that $\Hom^2(\dex s\la n\ra,R) = 0$, so we may invoke Proposition~\ref{prop:constr-proj} to obtain a projective cover $P$ of $\Sigma_t\la m\ra$ in $\ax$.  

The key observation now is that every $\Psi$-standard object is also $\Xi$-standard.  (This would not have been the case if we had instead defined $\Psi$ by deleting a standard-maximal element of $\Xi$.)  Thus, we now see from the exact sequence~\eqref{eqn:constr-proj} that $P$ has a $\Xi$-standard filtration.

We must now establish~\eqref{eqn:fc-proj}.  If $(u,p) \in \Psi$, we already know that
\[
\Hom^d(P', \Sigma_u\la p\ra) = \Hom^d(\dex s\la n\ra, \Sigma_u\la p\ra) = 0
\]
for all $d \ge 1$, so it follows that $\Hom^d(P, \Sigma_u\la p\ra) = 0$ as well.  It remains to show that
\[
\Hom^d(P, \Sigma_s\la n\ra) = 0.
\]
Consider the exact sequence
\[
\Hom^{d-1}(P, Q_s\la n\ra) \to \Hom^d(P, \Sigma_s\la n\ra) \to \Hom^d(P, \na s\la n\ra).
\]
The first term is already known to vanish because $Q_s\la n\ra \in \axg {\Psi}$, and the last term vanishes by Corollary~\ref{cor:tde-na-vanish} because $P$ has $\Xi$-standard filtration.  Thus, we have $\Hom^d(P, \Sigma_s\la n\ra) = 0$, as desired.
\end{proof}

%--------------------------------------------------------------------------
\subsection{Main result}
\label{subsect:main}
%--------------------------------------------------------------------------

Given a set $\Xi \subset S \times \Z$, let $\dx \subset \fD$ denote the full triangulated subcategory generated by $\ax$.  In other words, $\dx$ is full subcategory consisting of objects $X$ all of whose cohomology objects $H^i(X)$ lie in $\ax$.  Note that $\ax$ is the heart of a bounded $t$-structure on $\dx$, so we have a realization functor $\real: \Db(\ax) \to \dx$.  Composition with the inclusion functor gives us a natural functor $\Db(\ax) \to \fD$.

\begin{thm}\label{thm:qexc-dereq}
For any convex set $\Xi \subset S \times \Z$, the natural functor $\Db(\ax) \to \fD$ is fully faithful.  In particular, the realization functor
\[
\real: \Db(\fA) \to \fD
\]
is an equivalence of categories.
\end{thm}
\begin{proof}
We first treat the case where $\Xi$ is finite.  For a fixed object $X \in \ax$, $\{\Ext^d(\cdot,X)\}_{d \ge 0}$ is a universal $\delta$-functor, and Proposition~\ref{prop:fc-proj} tells us that the $\delta$-functor $\{\Hom^d(\cdot,X)\}_{d \ge 0}$ is effaceable.  Since their $0$th parts agree, there is a canonical isomorphism of functors $\Ext^d(\cdot, X) \overset{\sim}{\to} \Hom^d(\cdot, X)$.  Therefore, the natural functor $D^b(\ax) \to \fD$ is full and faithful.

Now, suppose $\Xi$ is infinite.  For any finite convex subset $\Psi \subset \Xi$, consider the chain of maps
\[
\Ext^i_\Psi(X,Y) \overset{g}{\to} \Ext^i_\Xi(X,Y) \overset{h}{\to} \Hom^i(X,Y).
\]
We claim that $g$ and $h$ are both isomorphisms for all $i$.  This is clearly the case for $i = 0$ and $i = 1$.  Suppose, in fact, that it is known for $i = 0,1, \ldots, d-1$.  By Lemma~\ref{lem:ext-hom}, both $g$ and $h$ are injective for $i = d$.  We know by the previous paragraph that the composition $h \circ g$ is an isomorphism for all $i$, so it follows that $g$ and $h$ are isomorphisms as well.

We have shown that $\Ext^d_\Xi(X,Y) \simeq \Hom^d(X,Y)$ when $X,Y \in \axg {\Psi}$.  But given any two objects $X,Y \in \ax$, we know from Lemma~\ref{lem:exist-convex} that there exists a finite convex subset $\Psi \subset \Xi$ such that $X,Y \in \axg {\Psi}$.  It follows that $\Ext^d_\Xi(X,Y) \simeq \Hom^d(X,Y)$ for all $X,Y \in \ax$, so $D^b(\ax) \to \fD$ is full and faithful.
\end{proof}

%%%%%%%%%%%%%%%%%%%%%%%%%%%%%%%%%%%%%%%%%%%%%%%%%%%%%%%%%%%%%%%%%%%%%%%%%%%
\section{Notation for semisimple groups}
\label{sect:not-red}
%%%%%%%%%%%%%%%%%%%%%%%%%%%%%%%%%%%%%%%%%%%%%%%%%%%%%%%%%%%%%%%%%%%%%%%%%%%

%--------------------------------------------------------------------------
\subsection{Representations and varieties}
%--------------------------------------------------------------------------

As noted in Section~\ref{sect:prelim-ab}, we will work over a fixed algebraically closed field $\Bbbk$.  For an algebraic group $H$ over $\Bbbk$, let $\Rep(H)$ denote the category of rational representations of $H$, and $\Repf(H) \subset \Rep(H)$ the subcategory of finite-dimensional representations.  If $H \subset K$, we have the usual induction and restriction functors $\ind_H^K : \Rep(H) \to \Rep(K)$ and $\res^K_H: \Rep(K) \to \Rep(H)$.  We also use the derived functor $\Rind_H^K: \Db\Rep(H) \to \Db\Rep(K)$.  (Of course, $\res^K_H$ is exact.)

For any variety $X$ over $\Bbbk$, we write $\Bbbk[X]$ for the ring of regular functions on $X$.  If $H$ is an algebraic group acting on $X$, we denote by $\Coh^H(X)$ the abelian category of $H$-equivariant coherent sheaves on $X$.  For any $\cF \in \Db\Coh^H(X)$, its derived global sections $R\Gamma(\cF)$ may be regarded as an object of $\Db\Rep(H)$.

%--------------------------------------------------------------------------
\subsection{Graded objects}
%--------------------------------------------------------------------------

Let $\uRep(H)$ be the category of graded rational $H$-repre\-sent\-ations.  This is, of course, equivalent to $\Rep(H \times \Gm)$.  $\uRepf(H)$ is defined similarly.  For homogeneous components and Tate twists of objects of $\uRep(H)$, we retain the conventions introduced in Section~\ref{subsect:tate} for $\uVect_\Bbbk$.

Consider $\Bbbk\la m\ra$, the graded $H$-representation consisting of the trivial $H$-module concentrated in degree $m$.  We can regard this as an $(H \times \Gm)$-equivariant sheaf on $\mathrm{pt} = \Spec \Bbbk$.  If $X$ is a variety equipped with an action of $H \times \Gm$, we put $\cO_X\la m\ra = a^* \Bbbk\la m\ra$, where $a: X \to \mathrm{pt}$ is the constant map.  More generally, for any $\cF \in \Db\Coh^{H \times \Gm}(X)$, we put
\[
\cF\la m\ra = \cF \Lotimes \cO_X\la m \ra.
\]

We write $\Forg$ for any of the various functors that forget gradings.  In particular, for coherent sheaves, we have $\Forg: \Db\Coh^{H \times \Gm}(X) \to \Db\Coh^H(X)$.  

%--------------------------------------------------------------------------
\subsection{Reductive groups}
%--------------------------------------------------------------------------

Throughout the rest of the paper, $G$ will be a simply connected semisimple algebraic group over $\Bbbk$, and the characteristic of $\Bbbk$ will be assumed to be good for $G$.  Fix a Borel subgroup $B \subset G$ and a maximal torus $T \subset B$.  Let $U \subset B$ be the unipotent radical, and let $\fu$ be the Lie algebra of $U$.  Recall that there is a $T$-equivariant isomorphism of varieties
\begin{equation}\label{eqn:fake-exp}
e: \fu \to U.
\end{equation}
Let $\Lambda$ be the weight lattice of $T$.  We will think of $B$ as the ``negative'' Borel: we define $\Lambda^+ \subset \Lambda$ to be the set of dominant weights determined by declaring the weights of $T$ on $\fu$ to be the negative roots.  Let $W$ be the Weyl group, and let $w_0 \in W$ be the longest element.  For any $\lambda \in \Lambda$, let $\dom(\lambda)$ denote the unique dominant weight in the $W$-orbit of $\lambda$.

Let $\le$ denote the usual partial order on $\Lambda$.  That is, for $\mu, \lambda \in \Lambda$, we say that $\mu \le \lambda$ if $\lambda - \mu$ is a nonnegative integer linear combination of positive roots.  We also define a preorder $\unlhd$ on $\Lambda$ as follows: $\mu \unlhd \lambda$ if $\dom(\mu) \le \dom(\lambda)$.  Obviously, $\le$ and $\unlhd$ coincide on $\Lambda^+$.  

For any $\lambda \in \Lambda$, let $\Bbbk_\lambda$ denote the $1$-dimensional $T$-rep\-re\-sent\-ation of weight $\lambda$.  We may also regard this as a $B$-representation on which $U$ acts trivially.  For $\lambda \in \Lambda^+$, let $L(\lambda)$, $M(\lambda)$, and $N(\lambda)$ denote the simple module, Weyl module, and dual Weyl module, respectively, of highest weight $\lambda$.  These representations may sometimes be regarded as graded by placing them in degree $0$.

%--------------------------------------------------------------------------
\subsection{Nilpotent cone and Springer resolution}
%--------------------------------------------------------------------------

Let $\cN$ be the variety of nilpotent elements in the Lie algebra of $G$.  We will also work with the flag variety $\cB = G/B$ and the Springer resolution $\tcN = G \times^B \fu$.  All these varieties are acted on by $G$.  Let $\Gm$ act on $\cN$ by $(z,x) \mapsto z^2 x$, where $z \in \Gm$ and $x \in \cN$.  This action commutes with the action of $G$.  The same formula defines an action on $\fu$ commuting with that of $B$, and so an action on $\tcN$ commuting with that of $G$.  Finally, let $\Gm$ act trivially on $\cB$.  The obvious projection maps, which we denote
\[
\cN \overset{\pi}{\longleftarrow} \tcN \overset{p}{\longrightarrow} \cB,
\]
are both $(G \times \Gm)$-equivariant.  Our convention on the $\Gm$-action means that the graded rings $\Bbbk[\fu]$ and $\Bbbk[\cN]$ are concentrated in even, nonpositive degrees.  We do not endow $U$ with a $\Gm$-action, so the map~\eqref{eqn:fake-exp} only gives rise to an isomorphism
\begin{equation}\label{eqn:fake-exp-rep}
\Bbbk[U] \simto \Forg(\Bbbk[\fu])
\end{equation}
of \emph{ungraded} $T$-representations.

Any graded $B$-representation $V$ gives rise to a locally free $(G \times \Gm)$-equivariant sheaf on $\cB$, denoted $\cS(V)$.  Give a weight $\lambda \in \Lambda$, consider the object
\[
A(\lambda) = R\pi_* p^*\cS(\Bbbk_\lambda).
\]
This is called the \emph{Andersen--Jantzen sheaf} of weight $\lambda$.  (It is known~\cite[Theorem~2]{klt:fscb} that $A(\lambda)$ is actually a coherent sheaf, rather than a complex of sheaves, for $\lambda$ dominant, but we will not use this fact.)  For any $\lambda \in \Lambda$, let
\[
\cD_{\lhd \lambda}, \cD_{\unlhd \lambda} \subset \Db\Cohgm(\cN)
\]
be the full triangulated subcategories generated by all Tate twists of Andersen--Jantzen sheaves $A(\mu)\la n \ra$ with $\mu \lhd \lambda$ or $\mu \unlhd \lambda$, respectively.

Consider now the object in $\Db\uRep(G)$ given by $R\Gamma(A(\lambda)) \cong R\Gamma(p_*p^*\cS(\Bbbk_\lambda))$.  By the projection formula, we have $p_*p^*\cS(\Bbbk_\lambda) \cong \cS(\Bbbk[\fu] \otimes \Bbbk_\lambda)$, so
\begin{equation}\label{eqn:aj-ind}
R\Gamma(A(\lambda)) \cong \Rind_B^G( \Bbbk[\fu] \otimes \Bbbk_\lambda).
\end{equation}

The following lemma on representations of a Borel subgroup is certainly well-known, but we include a proof for completeness.

\begin{lem}\label{lem:wt-ext}
Let $\mu \in \Lambda$, and let $U$ and $V$ be rational  $B$-representations such that all weights of $U$ are${}\le \mu$ but no weights of $V$ are${}\le \mu$.  Then $\RHom(U,V) = 0$.
\end{lem}
\begin{proof}
Let $\lambda \in \Lambda$.  Note that $\ind_T^B \Bbbk_\lambda$ is an injective $B$-module, since induction takes injective modules to injective modules.  Since $B \cong T \ltimes U$, we have $\ind_T^B \Bbbk_\lambda \cong \Bbbk[U] \otimes \Bbbk_\lambda$.  By~\eqref{eqn:fake-exp-rep}, $\res^B_T \ind_T^B \Bbbk_\lambda \cong \res^B_T \Forg(\Bbbk[\fu]) \otimes \Bbbk_\lambda$.  The weights of $\Bbbk[\fu]$ are sums of positive roots, so we see that all weights of $\ind_T^B \Bbbk_\lambda$ are${}\ge \lambda$.

The $B$-module $V$ can be embedded in an injective module $I^0$ by taking a direct sum of copies of $\ind_T^B \Bbbk_\lambda$ as $\lambda$ varies over weights of $V$.  It follows from the preceding paragraph that every weight of $I$ is $\ge$ some weight of $V$.  More generally, we can extend this an injective resolution $(I^n)_{n \ge 0}$ of $V$ in which every weight of every term is $\ge$ some weight of $V$.

Since no weight of $V$ is${}\le \mu$, it follows that no weight of $I^n$ is${}\le \mu$ either, so $\Hom(U,I^n) = 0$ for all $n$.  Thus, $\RHom(U,V) = 0$.
\end{proof}

%--------------------------------------------------------------------------
\subsection{Perverse coherent sheaves}
\label{subsect:pcoh}
%--------------------------------------------------------------------------

Recall that $G$ acts on $\cN$ with finitely many orbits, and that each orbit has even dimension.  For an orbit $C$, let $\eta_C$ be its generic point, and let $i_C: \{\eta_C\} \hto C$ be the inclusion map.  An object $\cF \in \Db\Cohg(\cN)$ is said to be a \emph{perverse coherent sheaf} if the following two conditions hold:
\begin{align*}
H^i(i_C^*\cF) &= 0 &&\text{for all $\textstyle i > \frac{1}{2}\codim C$,} \\
H^i(i_C^!\cF) &= 0 &&\text{for all $\textstyle i < \frac{1}{2}\codim C$.}
\end{align*}
The second condition is equivalent to requiring that
\begin{align*}
H^i(i_C^*\D\cF) &= 0 &&\text{for all $\textstyle i > \frac{1}{2}\codim C$,}
\end{align*}
where $\D$ is the \emph{Serre--Grothendieck duality functor} given by
\[
\D = \cRHom({-}, \cO_\cN).
\]
In fact, the functor $\D$ can be defined using any \emph{equivariant dualizing complex}~\cite{bez:pc}.  The fact that $\cO_\cN$ is a dualizing complex is equivalent to the fact that it is Gorenstein, cf.~\cite[Theorem~5.3.2]{bk:fsmgr}.  There is a choice of shifts and Tate twists here; our normalization agrees with the convention of~\cite{bez:ctm} but not with that of~\cite{bez:qes}.

The category $\Pcohg(\cN)$ of perverse coherent sheaves has a number of features in common with the more familiar perverse constructible sheaves.  Key among these are that every object has finite length, and that the simple objects admit a characterization resembling that of intersection cohomology complexes.  Simple objects are classified by pairs $(C,\cV)$, where $\cV$ is an irreducible $G$-equivariant vector bundle on $C$.  The corresponding simple object will be denoted $\cIC(C,\cV)$.

The category $\Pcohgm(\cN)$ is defined in the same way as above.  (Recall that the orbits of $G \times \Gm$ coincide with those of $G$.)  The forgetful functor $\Forg: \Db\Cohgm(\cN) \to \Db\Cohg(\cN)$ restricts to an exact functor
\[
\Forg: \Pcohgm(\cN) \to \Pcohg(\cN)
\]
that takes simple objects of $\Pcohgm(\cN)$ to simple objects of $\Pcohg(\cN)$.

%%%%%%%%%%%%%%%%%%%%%%%%%%%%%%%%%%%%%%%%%%%%%%%%%%%%%%%%%%%%%%%%%%%%%%%%%%%
\section{Andersen--Jantzen sheaves and perverse coherent sheaves}
\label{sect:aj}
%%%%%%%%%%%%%%%%%%%%%%%%%%%%%%%%%%%%%%%%%%%%%%%%%%%%%%%%%%%%%%%%%%%%%%%%%%%

In this section, we prove a number of lemmas on Andersen--Jantzen sheaves.  We work with $(G \times \Gm)$-equivariant sheaves throughout.  For the most part, the $\Gm$-action will play no essential role; nearly every statement in this section has an obvious $G$-equivariant analogue, with the same proof.  The only exception to this is part~\eqref{it:qexc-mixed} of Proposition~\ref{prop:qexc}, whose statement and proof involve imposing conditions on Tate twists.

Many proofs in this section are closely modeled on those in~\cite[Section~3]{bez:qes}, suitably modified to handle the difficulties that arise in positive characteristic.

\begin{lem}\label{lem:aj-sgd}
For all $\lambda \in \Lambda$, we have $\D A(\lambda) \cong A(-\lambda)$.
\end{lem}
\begin{proof}
Recall that proper pushforward $R\pi_*$ commutes with Serre--Grothen\-dieck duality, where the duality functor on $\tcN$ is given by $\D_{\tcN} = \cRHom({-}, \pi^!\cO_\cN)$.  It is a consequence of~\cite[Lemma~3.4.2 and Lemma~5.1.1]{bk:fsmgr} that $\pi^!\cO_\cN \cong \cO_{\tcN}$, so
\[
\D_{\tcN}(p^*\cS(\Bbbk_\lambda)) \cong \cRHom(p^*\cS(\Bbbk_\lambda), \cO_{\tcN}) \cong p^*\cS(\Bbbk_{-\lambda}),
\]
and the lemma follows.
\end{proof}

\begin{lem}\label{lem:aj-pcoh}
For all $\lambda \in \Lambda$, we have $A(\lambda) \in \Pcohgm(\tcN)$.  In particular, for all $\mu, \lambda \in \Lambda$, we have $\uHom^i(A(\mu), A(\lambda)) = 0$ if $i < 0$.
\end{lem}
\begin{proof}
Recall that the Springer resolution is semismall~\cite[Theorem~10.11]{jan:nort}.  This means that for any closed point $x$ in an orbit $C \subset \cN$, we have $\dim \pi^{-1}(x) \le \frac{1}{2}\codim C$.  Let $U_C$ be the union of the nilpotent orbits whose closure contains $C$.  Thus, $U_C$ is an open $G$-stable subset of $\cN$, and $C$ is the unique closed orbit therein.  Let $\tcN_C = \pi^{-1}(U_C)$.  Every fiber of the proper map $\pi: \tcN_C \to U_C$ has dimension${}\le \frac{1}{2}\codim C$, so by~\cite[Corollary~III.11.2]{har:ag}, it follows that $R^i\pi_*(p^*\cS(\Bbbk_\lambda))|_{\tcN_U} = 0$ for $i > \frac{1}{2}\codim C$.  Since $i_C^*$ is an exact functor (where $i_C: \{\eta_C\} \hto \cN$ is as in Section~\ref{subsect:pcoh}), it follows that $H^i(i_C^*A(\lambda)) = 0$ for $i > \frac{1}{2}\codim C$.  The same reasoning applies to $A(-\lambda) \cong \D A(\lambda)$, so $A(\lambda) \in \Pcohgm(\cN)$.

The last assertion of the lemma is just the general fact that $\Hom^i(X,Y)$ always vanishes for $i < 0$ if $X$ and $Y$ are in the heart of some $t$-structure.
\end{proof}

\begin{lem}\label{lem:demazure}
Let $\lambda, \mu \in \Lambda$ be two weights in the same $W$-orbit.  If $\mu \le \lambda$, then
\[
A(\mu) \in \cD_{\lhd \lambda} * A(\lambda)\la -2\ell\ra,
\]
where $\ell$ is the length of the shortest element $w \in W$ such that $w \lambda = \mu$.
\end{lem}
\begin{proof}
The statement is trivial if $\mu = \lambda$, so assume that $\mu < \lambda$.  It is easily seen by induction on $\ell$ that it suffices to prove this in the case where $\mu = s\lambda$ for some simple reflection $s$, say corresponding to the simple root $\alpha$.  Let $n = \la \alpha^\vee, \lambda\ra$.  Since $s\lambda < \lambda$, we have $n > 0$.

Let $P_\alpha \subset G$ be the minimal parabolic subgroup corresponding to $\alpha$, and let $p_\alpha: G/B \to G/P_\alpha$ be the projection map.  Let $\rho = \frac{1}{2} \sum \alpha$, where the sum runs over all positive roots.  Recall that $G$ is assumed to be simply connected, so $\rho$ lies in the weight lattice for $G$.  Let $Q = \Bbbk_{\rho-\alpha} \otimes \res^{P_\alpha}_B \ind_B^{P_\alpha} \Bbbk_{\lambda - \rho}$.  Since $\la \alpha^\vee, \lambda - \rho \ra = n - 1$, the weights of $\ind_B^{P_\alpha} \Bbbk_{\lambda - \rho}$ are $\lambda - \rho, \lambda - \rho - \alpha, \ldots, \lambda - \rho - (n-1)\alpha$.  Thus, the weights of $Q$ are
\[
\lambda - \alpha, \lambda - 2\alpha, \ldots, \lambda - n\alpha = s\lambda.
\]
A standard fact relating induction, restriction, and tensor products tells us that
\begin{equation}\label{eqn:rindq}
\Rind_B^{P_\alpha} Q \cong \Rind_B^{P_\alpha} \Bbbk_{\rho - \alpha} \Lotimes \Rind_B^{P_{\alpha}} \Bbbk_{\lambda - \rho} = 0,
\end{equation}
where the last equality follows from the fact that $\la\alpha^\vee, \rho - \alpha\ra = -1$.

From the weights of $Q$, we see that there is a short exact sequence of $B$-modules
\[
0 \to \Bbbk_{s\lambda} \to Q \to K_1 \to 0,
\]
where the weights of $K_1$ are${}\lhd \lambda$.  Applying $R\pi_* \circ p^* \circ \cS$, we see that
\begin{equation}\label{eqn:dem1}
A(s\lambda) \in \cD_{\lhd \lambda} * R\pi_*p^*\cS(Q).
\end{equation}
Similarly, there is a short exact sequence
\[
0 \to K_2 \to Q \otimes \Bbbk_\alpha \to \Bbbk_\lambda \to 0
\]
where $K_2$ has weights that are${}\lhd \lambda$.  We deduce that
\begin{equation}\label{eqn:dem2}
R\pi_*p^*\cS(Q \otimes \Bbbk_\alpha) \in \cD_{\lhd \lambda} * A(\lambda).
\end{equation}
In view of~\eqref{eqn:dem1} and~\eqref{eqn:dem2}, we see that the lemma will follow once we prove that
\begin{equation}\label{eqn:dem3}
R\pi_*p^*\cS(Q \otimes \Bbbk_\alpha)\la -2\ra \cong R\pi_*p^*\cS(Q).
\end{equation}

Let $\fu_\alpha$ be the Lie algebra of the unipotent radical of $P_\alpha$, and consider its coordinate ring $\Bbbk[\fu_\alpha]$.  It is the quotient of the graded ring $\Bbbk[\fu]$ by the ideal generated by $\alpha \in \fu^* = (\Bbbk[\fu])_{-2}$.  In other words, we have a short exact sequence
$
0 \to \Bbbk[\fu] \otimes \Bbbk_\alpha\la-2\ra \to \Bbbk[\fu] \to \Bbbk[\fu_\alpha] \to 0
$
of $(B \times \Gm)$-equivariant $\Bbbk[\fu]$-modules, or equivalently of objects in $\Coh^{B \times \Gm}(\fu)$.  A construction analogous to that of $\cS$ then gives us a short exact sequence
\[
0 \to p^*\cS(\Bbbk_\alpha)\la-2\ra \to \cO_{\tcN} \to i_*\cO_{\tcN_\alpha} \to 0
\]
in $\Cohgm(\tcN)$.  Here $\tcN_\alpha = G \times^B \fu_\alpha$, and $i: \tcN_\alpha \to \tcN$ is the inclusion map.  Tensoring with $p^*\cS(Q)$, we see that~\eqref{eqn:dem3} would follow if we knew that
\begin{equation}\label{eqn:dem4}
R\pi_*(i_*\cO_{\tcN_\alpha} \otimes p^*\cS(Q)) = 0.
\end{equation}
Since $\cN$ is an affine variety, $R\Gamma$ kills no nonzero object of $\Db\Cohgm(\cN)$, so it suffices to check that the object $R\Gamma(R\pi_*(i_*\cO_{\tcN_\alpha} \otimes p^*\cS(Q))) \cong R\Gamma(Rp_*(i_*\cO_{\tcN_\alpha} \otimes p^*\cS(Q)))$ vanishes.  By the projection formula. we have $Rp_*(i_*\cO_{\tcN_\alpha} \otimes p^*\cS(Q)) \cong \cS(\Bbbk[\fu_\alpha] \otimes Q)$, so to prove~\eqref{eqn:dem3}, we must check that $R\Gamma(\cS(\Bbbk[\fu_\alpha] \otimes Q)) = 0$, or
\begin{equation}\label{eqn:dem5}
\Rind_B^G (\Bbbk[\fu_\alpha] \otimes Q) = 0.
\end{equation}
But $\Rind_B^{P_{\alpha}}(\Bbbk[\fu_\alpha] \otimes Q) \cong \Bbbk[\fu_\alpha] \Lotimes \Rind_B^{P_\alpha} Q$, so~\eqref{eqn:dem5} follows from~\eqref{eqn:rindq}.
\end{proof}

\begin{lem}\label{lem:ratsing}
For any $\mu \in \Lambda^+$, we have $R\pi_*p^*\cS(M(\mu)) \cong \cO_\cN \otimes M(\mu)$.  Moreover, there are weights $\nu_1, \ldots, \nu_k$ such that
\begin{equation}\label{eqn:ratsing}
\cO_\cN \otimes M(\mu) \in A(\nu_1) * \cdots * A(\nu_k) * A(\mu)
\end{equation}
where either $\nu_i \lhd \mu$ or $\nu_i \in W\mu$ but $\nu_i \ne \mu$ for each $i$.

As a consequence, $\cO_\cN \otimes M(\mu) \in \Pcohgm(\cN)$, and there is a surjective map $M(\mu) \to A(\mu)$.
\end{lem}
\begin{proof}
Since $\pi^*(\cO_\cN \otimes M(\mu)) \cong p^*\cS(M(\mu))$, the projection formula implies that $R\pi_*p^*\cS(M(\mu)) \cong R\pi_*\cO_{\tcN} \Lotimes (\cO_\cN \otimes M(\mu))$.  But $R\pi_*\cO_{\tcN} \cong \cO_\cN$ by~\cite[Theorem~5.3.2]{bk:fsmgr}, so $R\pi_*p^*\cS(M(\mu)) \cong \cO_\cN \otimes M(\mu)$.

Next, there is a surjective map of $B$-representations $\res^G_B M(\mu) \to \Bbbk_\mu$.  The kernel of this map has a filtration whose subquotients are various $\Bbbk_{\nu_1}, \ldots, \Bbbk_{\nu_k}$, where either $\nu_i \lhd \mu$ or $\nu_i \in W\mu$ and $\nu < \mu$.  Applying the functor $R\pi_*p^*\cS$, we see that~\eqref{eqn:ratsing} holds. It now follows from Lemma~\ref{lem:aj-pcoh} that $R\pi_*p^*\cS(M(\mu)) \in \Pcohgm(\cN)$.  In particular, there is a distinguished triangle
\[
\cK \to \cO_\cN \otimes M(\mu) \to A(\mu) \to
\]
with $\cK \in A(\nu_1) * \cdots * A(\nu_k)$.  Since all three terms belong to $\Pcohgm(\cN)$, this is actually a short exact sequence in that category, and the map $\cO_\cN \otimes M(\mu) \to A(\mu)$ is surjective.
\end{proof}

\begin{lem}\label{lem:weyl-aj}
Let $\lambda, \mu \in \Lambda^+$.
\begin{enumerate}
\item If $\lambda \not \leq \mu$, then $\uRHom(\cO_\cN \otimes M(\mu), A(\lambda)) = 0$.\label{it:maj-vanish}
\item If $\lambda = \mu$, then $\uRHom(\cO_\cN \otimes M(\mu), A(\lambda)) \cong \Bbbk$. \label{it:maj-eq}
\end{enumerate}
\end{lem}
\begin{proof}
We have $\uRHom(\cO_\cN \otimes M(\mu), A(\lambda)) \cong \uRHom(M(\mu), R\Gamma(A(\lambda)))$ by adjunction.  We will work with the latter object.  Using~\eqref{eqn:aj-ind}, we have
\[
\uRHom_G(M(\mu), R\Gamma(A(\lambda))) \cong \uRHom_B(\res^G_B M(\mu), \Bbbk[\fu] \otimes \Bbbk_\lambda).
\]
Of course, all weights of $M(\mu)$ are${}\le \mu$, and all weights of $\Bbbk[\fu] \otimes \Bbbk_\lambda$ are${}\ge \lambda$.  Part~\eqref{it:maj-vanish} then follows from Lemma~\ref{lem:wt-ext}.

For part~\eqref{it:maj-eq}, let $J \subset \Bbbk[\fu]$ be the ideal spanned by all homogeneous elements of strictly negative degree.  Thus, $\Bbbk[\fu] \otimes \Bbbk_\mu \cong \Bbbk_\mu \oplus (J \otimes \Bbbk_\mu)$.  Since all weights of $J \otimes \Bbbk_\mu$ are${}> \mu$, Lemma~\ref{lem:wt-ext} again tells us that $\uRHom(\res^G_B M(\mu), J \otimes \Bbbk_\mu) = 0$.  We conclude that
\begin{multline*}
\RHom_G(M(\mu), R\Gamma(A(\mu)\la n\ra)) \cong \RHom_B(\res^G_B M(\mu), \Bbbk_\mu\la n\ra) \\
\cong \RHom_G(M(\mu), \Rind_B^G \Bbbk_\mu\la n\ra) \cong \RHom_G(M(\mu), N(\mu)\la n\ra),
\end{multline*}
and this clearly vanishes for $n \ne 0$ and is $1$-dimensional when $n = 0$.
\end{proof}

\begin{prop}\label{prop:qexc}
Let $\lambda \in \Lambda^+$.  We have:
\begin{enumerate}
\item If $\mu \in \Lambda$ and $\mu \not\unrhd\lambda$, then $\uRHom(A(\mu), A(\lambda)) = 0$.\label{it:qexc-ord}
\item If $i < 0$, then $\uHom^i(A(\lambda),A(\lambda)) = 0$, and $\uHom(A(\lambda), A(\lambda)) \cong \Bbbk$.\label{it:qexc-self}
\item If $i > 0$ and $n \ge 0$, then $\Hom^i(A(\lambda), A(\lambda)\la n\ra) = 0$.\label{it:qexc-mixed}
\item If $\mu \in \Lambda^+$ and $\lambda \neq \mu$, then $\uRHom(A(w_0\mu), A(\lambda)) = 0$.\label{it:qexc-dual}
\end{enumerate}
\end{prop}
\begin{proof}
\eqref{it:qexc-ord} Fix $\lambda$.  In the proof, we will assume that $\mu$ is dominant and that $\lambda \not\le \mu$, and we will show that $\uRHom(A(w\mu), A(\lambda)) = 0$ for all $w \in W$.  We proceed by induction with respect to $\unlhd$.  Assume that for all $\nu \lhd \mu$, we already know that $\uRHom(A(\nu), A(\lambda)) = 0$.  We know that $\uRHom(\cO_\cN \otimes M(\mu), A(\lambda)) = 0$ by Lemma~\ref{lem:weyl-aj}\eqref{it:maj-vanish}.  On the other hand, applying $\uRHom({-},A(\lambda))$ to~\eqref{eqn:ratsing}, we have
\begin{multline*}
\uRHom(\cO_\cN \otimes M(\mu), A(\lambda)) \in \\
\uRHom(A(\mu), A(\lambda)) * \uRHom(A(\nu_k), A(\lambda)) * \cdots * \uRHom(A(\nu_1), A(\lambda)).
\end{multline*}
All terms on the right-hand side with $\nu_i \lhd \mu$ vanish by assumption and can be omitted.   The remaining terms are those with $\nu_i \in W\mu$.  By Lemma~\ref{lem:demazure} and the inductive assumption, there is some integer $n > 0$ such that
\[
\uRHom(A(\nu_i), A(\lambda)) \cong \uRHom(A(\mu), A(\lambda))\la n\ra \qquad
\text{if $\nu_i \in W\mu$.}
\]
Therefore, the expression above simplifies to
\begin{multline}\label{eqn:qexc-ord1}
\uRHom(\cO_\cN \otimes M(\mu), A(\lambda)) \in \\
\uRHom(A(\mu), A(\lambda)) * \uRHom(A(\mu), A(\lambda))\la n_1\ra * \cdots * \uRHom(A(\mu), A(\lambda))\la n_m\ra. 
\end{multline}
By Lemma~\ref{lem:selfstar}\eqref{it:selfstar-van}, we conclude that $\uRHom(A(\mu),A(\lambda)) = 0$.

\eqref{it:qexc-self} The first assertion of this part is contained in Lemma~\ref{lem:aj-pcoh}.  The second assertion will be proved together with part~\eqref{it:qexc-mixed} in the next paragraph.

\eqref{it:qexc-mixed} This proof is similar to that of part~\eqref{it:qexc-ord}.  We know from Lemma~\ref{lem:weyl-aj}\eqref{it:maj-eq} that $\uRHom(\cO_\cN \otimes M(\lambda), A(\lambda)) \cong \Bbbk$.  We may again carry out the calculations leading to~\eqref{eqn:qexc-ord1}, this time with $\mu = \lambda$.  Since $n_1, \ldots, n_m > 0$, Lemma~\ref{lem:selfstar}\eqref{it:selfstar-pos} tells us that $\uHom(A(\lambda),A(\lambda)) \cong \Bbbk$, and that for $i > 0$, $\uHom^i(A(\lambda), A(\lambda))$ is concentrated in strictly positive degrees.  In other words, for $n \ge 0$, $\Hom^i(A(\lambda), A(\lambda)\la n\ra) = 0$.

\eqref{it:qexc-dual} If $\mu \not\ge \lambda$, then this is an instance of part~\eqref{it:qexc-ord}.  On the other hand, if $\mu > \lambda$, then we apply Serre--Grothendieck duality and Lemma~\ref{lem:aj-sgd}:
\[
\uRHom(A(w_0\mu), A(\lambda)) \cong \uRHom(\D A(\lambda), \D A(w_0\mu)) \cong \uRHom(A(-\lambda), A(-w_0\mu)).
\]
Now, $-w_0\mu$ and $-w_0\lambda$ are both dominant, and $-w_0\mu > -w_0\lambda$.  In particular, we have $-\lambda \not\unrhd -w_0\mu$, so $\uRHom(A(-\lambda), A(-w_0\mu)) = 0$ by part~\eqref{it:qexc-ord} again.
\end{proof}

\begin{lem}\label{lem:cb-gen}
Let $\fC$ be the category of finitely-generated graded $B$-equivariant modules over the graded ring $\Bbbk[\fu]$.  Then $\Db\fC$ is generated as a triangulated category by objects of form $\Bbbk[\fu] \otimes V\la n\ra$, where $V$ is a finite-dimensional $B$-representation.
\end{lem}
\begin{proof}
In this proof, we will say that an object $M \in \fC$ is \emph{free} if it is a direct sum of objects of the form $\Bbbk[\fu] \otimes V\la n\ra$.  Let $R$ be the functor which forgets the $B$-action (but retains the grading).  Clearly, $R$ takes free objects of $\fC$ to free (graded) $\Bbbk[\fu]$-modules.  However, a module $M \in \fC$ may have the property that $R(M)$ is a free module while $M$ itself is not.  Let us call a module $M$\emph{weakly free} if $R(M)$ is free.

It is easy to see that $\fC$ has ``enough'' free objects, i.e., that every module is a quotient of a free module.  Therefore, every module $M$ has (possibly infinite) resolution by free modules $\cdots \to F_1 \to F_0 \to M \to 0$.  Hilbert's syzygy theorem, in the form found in, say,~\cite[Corollary~3.19]{clo:uag}, asserts that there is some $n$ such that the kernel of the map $F_n \to F_{n-1}$ is free as a graded $\Bbbk[\fu]$-module, i.e. weakly free. Thus, every module admits a finite resolution whose terms are either free or weakly free.  It follows that $\Db\fC$ is generated by the weakly free modules.

The lemma then follows from the following claim: \emph{Every weakly free module admits a finite filtration whose subquotients are free modules}.  Let $M$ be a weakly free module, and let $m_1, \ldots, m_n$ be a set of homogeneous elements that constitute a basis for it as a free $\Bbbk[\fu]$-module.  Let $N = \max \{ \deg m_i\}$, and assume without loss of generality that $m_1, \ldots, m_k$ have degree $N$ and that $m_{k+1}, \ldots, m_n$ have degree${}< N$.  Then $m_1, \ldots, m_k$ must constitute a $\Bbbk$-basis for the vector space $M_N$.  The $\Bbbk[\fu]$-submodule $M'$ generated by $m_1, \ldots, m_k$ is a free $\Bbbk[\fu]$-module and a direct summand of $R(M)$.  It is also stable under $B$ and isomorphic to $\Bbbk[\fu] \otimes M_N$ as an object of $\fC$.  In other words, $M'$ is a subobject of $M$ in $\fC$; it is free, and the quotient $M/M'$ is weakly free.  The claim then follows by induction on the rank of $R(M)$.
\end{proof} 

Via the equivalences $\fC \cong \Coh^{B \times \Gm}(\fu) \cong \Cohgm(\tcN)$, we obtain the following result.

\begin{cor}\label{cor:tcn-gen}
$\Db\Cohgm(\tcN)$ is generated as a triangulated category by the objects of the form $p^*\cS(V)\la n\ra$, where $V$ ranges over all finite-dimensional $B$-representations. \qed
\end{cor}

\begin{lem}\label{lem:cn-gen}
$\Db\Cohgm(\cN)$ is generated as a triangulated category by objects of the form $R\pi_*\cF$, where $\cF \in \Db\Cohgm(\tcN)$.
\end{lem}
\begin{proof}
Let $\fD \subset \Db\Cohgm(\cN)$ be the subcategory generated by objects $R\pi_*\cF$ for $\cF \in \Db\Cohgm(\tcN)$.  Because $\Pcohgm(\cN)$ is a finite-length category that is the heart of a bounded $t$-structure, we have that the simple perverse coherent sheaves generate $\Db\Cohgm(\cN)$ as a triangulated category, so it suffices to show that the simple perverse coherent sheaves lie in $\fD$.

Consider a simple perverse coherent sheaf $\cIC(C, \cV)$, where $C \subset \cN$ is a nilpotent orbit, and $\cV$ is an irreducible $G$-equivariant vector bundle on $C$.  Let $Z = \overline{C} \smallsetminus C$.  We proceed by induction on $C$ with respect to the closure partial order on nilpotent orbits.  That is, we assume that $\cIC(C',\cV') \in \fD$ for all $C' \subset Z$.  The latter objects generate the full triangulated subcategory $\Db_Z\Cohgm(\cN) \subset \Db\Cohgm(\cN)$ consisting of objects whose support is contained in $Z$.  Thus, our assumption implies that $\Db_Z\Cohgm(\cN) \subset \fD$.

By\cite[Proposition~5.9 and~8.8(II)]{jan:nort}, there is a parabolic subgroup $P \supset B$ and a $P$-stable subspace $\fv \subset \fu \cap \overline{C}$ such that the natural map $q: G \times^P \fv \to \overline{C}$ is a resolution of singularities of $\overline{C}$.  Consider the variety $X = G \times^B \fv$.  We have an inclusion $\tilde\imath: X \to \tcN$, as well as an obvious smooth map $h: X \to G \times^P \fv$ whose fibers are isomorphic to $P/B$.  Let $i: \overline{C} \to \tcN$ be the inclusion map.

Let $\cG \in \Db\Cohgm(\overline{C})$ be an object such that $i_*\cG \cong \cIC(C,\cV)$.  (Because coherent pullback is not exact, some care must be taken to distinguish between these two objects.)  Let $\tilde\cG = (q \circ h)^*\cG$, and let $\cF = \tilde\imath_*\tilde\cG$.  Since $R\Gamma(P/B, \cO_{P/B}) \cong \Bbbk$, it follows from the projection formula that the canonical adjunction morphism $q^*\cG \simto Rh_*h^*(q^*\cG)$ is an isomorphism.  Applying $Rq_*$, we obtain an isomorphism $Rq_*q^*\cG \to R(q \circ h)_*\tilde\cG$.  Then, composing with $\cG \to Rq_*q^*\cG$, we get a morphism
\[
\cG \to R(q \circ h)_*\tilde\cG.
\]
This map is at least an isomorphism over $C$, since $q$ is an isomorphism over $C$.  Thus, its cone $\cK$ has support contained in $Z$.  Applying $i_*$, we have a distinguished triangle
\[
\cIC(C,\cV) \to R\pi_*\cF \to i_*\cK \to.
\]
Since $R\pi_*\cF \in \fD$ and $i_*\cK \in \Db_Z\Cohgm(\cN) \subset \fD$, we conclude that $\cIC(C,\cV) \in \fD$, as desired.
\end{proof}

%%%%%%%%%%%%%%%%%%%%%%%%%%%%%%%%%%%%%%%%%%%%%%%%%%%%%%%%%%%%%%%%%%%%%%%%%%%
\section{Proofs of the main results}
\label{sect:proofs}
%%%%%%%%%%%%%%%%%%%%%%%%%%%%%%%%%%%%%%%%%%%%%%%%%%%%%%%%%%%%%%%%%%%%%%%%%%%

The results of Section~\ref{sect:aj} fit the framework of Sections~\ref{sect:prelim-ab}--\ref{sect:der-eq} and allow us to quickly deduce the main results.  For $\lambda \in \Lambda^+$, let $\delta_\lambda$ denote the length of the shortest element $w \in W$ such that $w\lambda = w_0\lambda$.  We then put
\begin{equation}\label{eqn:pcoh-qh}
\begin{aligned}
\na \lambda &= A(\lambda)\la -\delta_\lambda\ra, \\
\de \lambda &= A(w_0\lambda)\la \delta_\lambda\ra
\end{aligned}
\end{equation}

\begin{prop}
The objects $\na \lambda$ constitute an abelianesque dualizable graded quasi-exceptional set in $\Db\Cohgm(\cN)$, and the $\de \lambda$ form the dual set.

Likewise, the objects $\Forg(\de \lambda)$ constitute an abelianesque dualizable ungraded quasi-exceptional set in $\Db\Cohg(\cN)$, and the $\Forg(\de \lambda)$ form the dual set.
\end{prop}
\begin{proof}
Referring to Definition~\ref{defn:qexc}, we see that conditions~\eqref{it:defq-ord}--\eqref{it:defq-self} are proved in Proposition~\ref{prop:qexc}.  To see that condition~\eqref{it:defq-gen} holds, note that every graded finite-dimensional $B$-representation arises by extensions among $1$-dimensional representations $\Bbbk_\lambda\la n\ra$.  By Corollary~\ref{cor:tcn-gen}, the objects $p^*\cS(\Bbbk_\lambda\la n\ra)$ generate $\Db\Cohgm(\tcN)$, and then by Lemma~\ref{lem:cn-gen}, the objects $A(\lambda)\la n\ra$, where $\lambda \in \Lambda$ and $n \in \Z$, generate $\Db\Cohgm(\cN)$.  The fact that it suffices to take the $A(\lambda)\la n\ra$ with $\lambda$ dominant follows from Lemma~\ref{lem:demazure} with an induction argument with respect to $\unlhd$.  Thus, the $\{A(\lambda)\la \delta_\lambda\ra\}$ with $\lambda \in \Lambda^+$ form a graded quasi-exceptional set.

In fact, the aforementioned induction argument also shows that each $\cD_{\lhd \lambda}$ is generated by the $A(\mu)\la n\ra$ with $\mu \in \Lambda^+$, $\mu < \lambda$.  So this category coincides with the one that would have been denoted $\dqt \lambda$ in Section~\ref{sect:prelim-ab}.  By Lemma~\ref{lem:demazure}, there is a morphism $A(w_0\lambda) \to A(\lambda)\la -2\delta_\lambda\ra$ whose cone lies in $\cD_{\lhd \lambda}$.  Combining this observation with Proposition~\ref{prop:qexc}\eqref{it:qexc-dual}, we see that the $\{A(w_0\lambda)\la \delta_\lambda\ra\}$ forms a dual set.  The fact that it is abelianesque is contained in Lemma~\ref{lem:aj-pcoh}.

For the ungraded version, we omit part~\eqref{it:defq-mixed} of Definition~\ref{defn:qexc}.  Since Proposition~\ref{prop:qexc}\eqref{it:qexc-mixed} was the only result of Section~\ref{sect:aj} without an ungraded analogue (see the remarks at the beginning of Section~\ref{sect:aj}), the ungraded version of the present proposition also holds.
\end{proof}

\begin{thm}
The categories $\Pcohgm(\cN)$ and $\Pcohg(\cN)$ are quasi-heredi\-tary, with standard and costandard objects as in~\eqref{eqn:pcoh-qh}.
\end{thm}
\begin{proof}
By Theorem~\ref{thm:qexc-t}, the objects $\de \lambda$ and $\na \lambda$ determine a $t$-structure on each of $\Db\Cohgm(\cN)$ and $\Db\Cohg(\cN)$ whose heart $\fA$ is quasi-hereditary and in which those objects are standard and costandard, respectively.  But it is easily seen from Lemma~\ref{lem:aj-pcoh} and the definition given in Theorem~\ref{thm:qexc-t} that every perverse coherent sheaf lies in $\fA$.  The heart of one bounded $t$-structure cannot be properly contained in the heart of another bounded $t$-structure, so it must be that $\fA$ coincides with $\Pcohgm(\tcN)$ or $\Pcohg(\cN)$.
\end{proof}

\begin{thm}
The functor $\real: \Db\Pcohgm(\cN) \simto \Db\Cohgm(\cN)$ is an equivalence of categories.
\end{thm}
\begin{proof}
By Lemma~\ref{lem:weyl-aj}\eqref{it:maj-eq}, we have $\uHom^d(\cO_\cN \otimes M(\lambda), A(\lambda)) = 0$ for all $d > 0$.  It follows that if $X \in \Pcohgm(\cN)$ is a $\lambda$-quasicostandard object, then $\uHom^d(\cO_\cN \otimes M(\lambda), X) = 0$.  By Lemma~\ref{lem:ratsing}, we have a surjective map $\cO_\cN \otimes M(\lambda) \to A(\lambda)$.  Thus, part~\eqref{it:defeff-costd} of Definition~\ref{defn:efface} holds.  By Lemma~\ref{lem:aj-sgd}, the Serre--Grothendieck duality functor exchanges standard and costandard objects, so part~\eqref{it:defeff-std} of Definition~\ref{defn:efface} follows from part~\eqref{it:defeff-costd}.  By Theorem~\ref{thm:qexc-dereq}, the desired equivalence holds.
\end{proof}

%%%%%%%%%%%%%%%%%%%%%%%%%%%%%%%%%%%%%%%%%%%%%%%%%%%%%%%%%%%%%%%%%%%%%%%%%%%


\begin{thebibliography}{JMW}

\bibitem[A]{a:ekt}
P.~Achar, {\em On the equivariant $K$-theory of the nilpotent cone for the
  general linear group}, Represent. Theory {\bf 8} (2004), 180--211.

\bibitem[AR]{ar:kdsf}
P.~Achar and S.~Riche, {\em Koszul duality and semisimplicity of Frobenius},
  arXiv:1102.2820, submitted.

\bibitem[Be]{bei:dcps}
A.~Be{\u\i}linson, {\em On the derived category of perverse sheaves},
  {$K$}-theory, arithmetic and geometry ({M}oscow, 1984--1986), Lecture Notes
  in Mathematics, vol. 1289, Springer-Verlag, Berlin, 1987, pp.~27--41.

\bibitem[BBD]{bbd}
A.~Be{\u\i}linson, J.~Bernstein, and P.~Deligne, {\em Faisceaux pervers},
  Analyse et topologie sur les espaces singuliers, I (Luminy, 1981),
  Ast\'erisque, vol. 100, Soc. Math. France, Paris, 1982, pp.~5--171.

\bibitem[BGS]{bgs}
A.~Be{\u\i}linson, V.~Ginzburg, and W.~Soergel, {\em Koszul duality patterns in
  representation theory}, J. Amer. Math. Soc. {\bf 9} (1996), 473--527.

\bibitem[B1]{bez:pc}
R.~Bezrukavnikov, {\em Perverse coherent sheaves (after Deligne)},
  arXiv:math.AG/0005152.
  
\bibitem[B2]{bez:qes}
R.~Bezrukavnikov, {\em Quasi-exceptional sets and equivariant coherent sheaves
  on the nilpotent cone}, Represent. Theory {\bf 7} (2003), 1--18.

\bibitem[B3]{bez:ctm}
R.~Bezrukavnikov, {\em Cohomology of tilting modules over quantum groups and
  {$t$}-structures on derived categories of coherent sheaves}, Invent. Math.
  {\bf 166} (2006), 327--357.

\bibitem[B4]{bez:psaf}
R.~Bezrukavnikov, {\em Perverse sheaves on affine flags and nilpotent cone of
  the Langlands dual group}, Israel J. Math. {\bf 170} (2009), 185--206.

\bibitem[BK]{bk:fsmgr}
M.~Brion and S.~Kumar, {\em Frobenius splitting methods in geometry and representation
  theory}, Progr. Math., vol. 231, Birkh\"auser Boston, Boston, MA, 2005.

\bibitem[CLO]{clo:uag}
D.~A. Cox, J.~Little, and D.~O'Shea, {\em Using algebraic geometry}, 2nd ed.,
  Graduate Texts in Mathematics, vol. 185, Springer, New York, 2005.

\bibitem[F]{fie:sasv}
P.~Fiebig, {\em Sheaves on affine Schubert varieties, modular representations,
  and Lusztig's conjecture}, J. Amer. Math. Soc. {\bf 24} (2011), 133--181.

\bibitem[H]{har:ag}
R.~Hartshorne, {\em Algebraic geometry}, Graduate Texts in Mathematics, no.~52,
  Springer-Verlag, New York, 1977.

\bibitem[Ja]{jan:nort}
J.~C. Jantzen, {\em Nilpotent orbits in representation theory}, Lie theory,
  Progr. Math., vol. 228, Birkh\"auser Boston, Boston, MA, 2004, pp.~1--211.

\bibitem[Ju]{jut:mscdm}
D.~Juteau, {\em Modular Springer correspondence and decomposition matrices},
  Ph.D. thesis, Universit\'e Paris 7, 2007.

\bibitem[JMW]{jmw:psmrt}
D.~Juteau, C.~Mautner, and G.~Williamson, {\em Perverse sheaves and modular
  representation theory}, Geometric methods in representation theory II,
  S\'emin. Congr., vol.~25, Soc. Math. France, 2010, pp.~313--350.

\bibitem[KLT]{klt:fscb}
S.~Kumar, N.~Lauritzen, and J.~F. Thomsen, {\em Frobenius splitting of
  cotangent bundles of flag varieties}, Invent. Math. {\bf 136} (1999),
  603--621.

\bibitem[S]{soe:ricrt}
W.~Soergel, {\em On the relation between intersection cohomology and
  representation theory in positive characteristic}, J. Pure Appl. Algebra {\bf
  152} (2000), 311--335.

\end{thebibliography}
\end{document}